\documentclass[a4paper, 11pt]{article}

\usepackage[titletoc,toc,title]{appendix}

\input xy   
\xyoption{all}
\usepackage{amsmath, amsthm, amsfonts, amssymb}
\usepackage{enumerate}
\usepackage{stackrel}
\usepackage{colordvi}
\usepackage{hyperref}
\usepackage[english]{babel}
\usepackage{color}
\usepackage{verbatim}
\usepackage[all]{xy}
\usepackage{graphicx}
\usepackage{mathrsfs}

\DeclareMathAlphabet{\mathcalligra}{T1}{calligra}{m}{n}
\DeclareMathAlphabet{\mathscrmin}{T1}{mathscr}{m}{n}

\numberwithin{equation}{subsection}
\theoremstyle{plain}
  \begingroup
        \newtheorem{theorem}[equation]{Theorem}
        \newtheorem{lemma}[equation]{Lemma}
        \newtheorem{proposition}[equation]{Proposition}
        \newtheorem{corollary}[equation]{Corollary}

	    \newtheorem{definition}[equation]{Definition}
        
        \newtheorem{sinnadaitalica}[equation]{}

\endgroup

\usepackage[textwidth=0.89in,textsize=scriptsize]{todonotes}

\usepackage{amsmath,amssymb,amsthm, xspace,a4wide,color}
\theoremstyle{definition}
  \begingroup
        \newtheorem{remark}[equation]{Remark}
        
        \newtheorem{sinnadastandard}[equation]{}

	\newtheorem{notation}[equation]{Notation}
\endgroup

\newcommand{\Lbl}{\cc{L}}

\newcommand{\cqd}{\hfill$\Box$}

\newcommand{\cc}{\mathcal}

\newcommand{\ff}{\mathsf}

\newcommand{\de}{definition}
\newcommand{\prop}{proposition}
\newcommand{\A}{\mathcal{A}}
\newcommand{\B}{\mathcal{B}}
\newcommand{\C}{\mathcal{C}}

\newcommand{\E}{\mathcal{E}}
\newcommand{\Cat}{\mathcal{C}at}
\newcommand{\eps}{\varepsilon}

\newcommand{\mr}[1]{\overset {#1} {\longrightarrow}}

\newcommand{\xr}[1]{\xrightarrow {#1}}

\newcommand{\Mr}[1]{\overset {#1} {\Longrightarrow}}

\newcommand{\rimply}{\Rightarrow}

\newcommand{\mrpairviejo}[2]
   {
    \xymatrix@C=5ex@R=2.4ex
            {
             {} \ar@<1.6ex>[r]^{#1} 
	            \ar@<-1.1ex>[r]^{#2} 
	         & {}
            }
   }

\newcommand{\mrpair}[2]
   {
    \xymatrix@C=5ex@R=2.4ex
            {
             {} \ar@<1ex>[r]^{#1} 
	            \ar@<-1ex>[r]_{#2} 
	         & {}
            }
   }

 \newcommand{\mrpairc}[2]
   {
    \xymatrix@C=5ex@R=2.4ex
            {
             {} \ar@<1ex>[r]^{#1} 
	            \ar@<-1ex>[r]|{o}_{#2} 
	         & {}
            }
   }

 \newcommand{\mrpaircc}[2]
   {
    \xymatrix@C=5ex@R=2.4ex
            {
             {} \ar@<1ex>[r]|{o}^{#1} 
	            \ar@<-1ex>[r]|{o}_{#2} 
	         & {}
            }
   }

\newcommand{\mlpair}[2]
   {
    \xymatrix@C=5ex@R=2.4ex
            {
             {} 
              & {} \ar@<1.0ex>[l]_{#2} 
	          \ar@<-1.7ex>[l]_{#1}
            }
    }

\newcommand{\cellrd}[3] 
{
  \xymatrix@C=7ex@R=2.4ex
         {
          {} \ar@<1.6ex>[r]^{#1} 
             \ar@{}@<-1.3ex>[r]^{\!\! {#2} \, \!\Downarrow}
             \ar@<-1.1ex>[r]_{#3} 
          & {}
         }
 }
 \newcommand{\modif}[3]
 {
  \xymatrix@C=7ex@R=2.4ex
         {
          {} \ar@<1.6ex>@{=>}[r]^{#1} 
             \ar@{}@<-1.3ex>@{=>}[r]^{\!\! {#2} \, \!\downarrow}
             \ar@<-1.1ex>[r]_{#3} 
          & {}
         }
 }
 \newcommand{\scellrd}[3] 
 {
  \xymatrix@C=4.5ex@R=2.4ex
         {
          {} \ar@<1.6ex>[r]^{#1}
             \ar@{}@<-1.3ex>[r]^{\!\! {#2} \, \!\Downarrow}
             \ar@<-1.1ex>[r]_{#3}
          & {}
         }
}

\newcommand{\cellld}[3] 
 {
  \xymatrix@C=6ex@R=2.4ex
         {
            {} 
          & {} \ar@<1.0ex>[l]^{#3} 
          \ar@{}@<-1.7ex>[l]^{\!\! {#2} \, \!\Downarrow}
	                                 \ar@<-1.7ex>[l]_{#1}
         }
 }

\newcommand{\cellpairrd}[4] 
 {
  \xymatrix@C=10ex@R=2.4ex
         {
          {} \ar@<1.6ex>[r]^{#1} 
             \ar@{}@<-1.3ex>[r]^{\!\! {#2} \, \!\Downarrow 
                                 \;\; {#3} \, \!\Downarrow }
             \ar@<-1.1ex>[r]_{#4} 
          & {}
         }
 }

\newcommand{\cellpairrdc}[4] 
 {
  \xymatrix@C=10ex@R=2.4ex
         {
          {} \ar@<1.6ex>[r]^{#1} 
             \ar@{}@<-1.3ex>[r]^{\!\! {#2} \, \!\Downarrow 
                                 \;\; {#3} \, \!\Downarrow }
             \ar@<-1.1ex>[r]|{o}_{#4} 
          & {}
         }
 }
 
\newcommand{\cellpairrdcc}[4] 
 {
  \xymatrix@C=10ex@R=2.4ex
         {
          {} \ar@<1.6ex>[r]|{o}^{#1} 
             \ar@{}@<-1.3ex>[r]^{\!\! {#2} \, \!\Downarrow 
                                 \;\; {#3} \, \!\Downarrow }
             \ar@<-1.1ex>[r]|{o}_{#4} 
          & {}
         }
 }

\newcommand{\coLimsinc}[2]
   {
    \underset{#1}{\underrightarrow{\ff{Lim}}}
    \; {#2}
   }

 \newcommand{\Limsinc}[2]
   {
    \underset{#1}{\underleftarrow{\ff{Lim}}}
    \; {#2}
   }

\newcommand{\coLim}[2]
   {
    \underset{#1}{\underrightarrow{\ff{\sigma Lim}}}
    \; {#2}
   }

\newcommand{\opcoLim}[2]
   {
    \underset{#1}{\underrightarrow{\ff{op \sigma Lim}}}
    \; {#2}
   }

\newcommand{\cart}   
 {
  \mathscr{C}
 }

\newcommand{\coLimconw}[2]
   {
    \underset{#1}{\underrightarrow{\ff{\sigma Lim^\Sigma}}}
    \; {#2}
   }   

\newcommand{\opcoLimconw}[2]
   {
    \underset{#1}{\underrightarrow{\ff{op \sigma Lim^\Sigma}}}
    \; {#2}
   }    
   
   \newcommand{\Limconw}[2]
   {
    \underset{#1}{\underleftarrow{\ff{\sigma Lim^\Sigma}}}
    \; {#2}
   }   
   
\newcommand{\coLimconcartsigma}[2]
   {
    \underset{#1}{\underrightarrow{\ff{\sigma Lim^{\cart_\Sigma}}}}
    \; {#2}
   }

\newcommand{\coLimconcartp}[2]
   {
    \underset{#1}{\underrightarrow{\ff{\sigma Lim^{\cart_P}}}}
    \; {#2}
   }
\newcommand{\coLimconcartw}[2]
   {
    \underset{#1}{\underrightarrow{\ff{\sigma Lim^{\cart_W}}}}
    \; {#2}
   }

\newcommand{\bicoLimconcartsigma}[2]
   {
    \underset{#1}{\underrightarrow{\ff{\sigma biLim^{\cart_\Sigma}}}}
    \; {#2}
   }       

\newcommand{\bicoLimconcartw}[2]
   {
    \underset{#1}{\underrightarrow{\ff{\sigma biLim^{\cart_W}}}}
    \; {#2}
   }       
\newcommand{\bicoLimconw}[2]
   {
    \underset{#1}{\underrightarrow{\ff{\sigma biLim^{\Sigma}}}}
    \; {#2}
   }

\newcommand{\Lim}[2]
   {
    \underset{#1}{\underleftarrow{\ff{\sigma Lim}}}
    \; {#2}
   }

\newcommand{\cbiLim}[2]
   {
    \underset{#1}{\underleftarrow{\ff{\sigma biLim}}}
    \; {#2}
   }   

\newcommand{\bicoLim}[2]
   {
    \underset{#1}{\underrightarrow{\ff{\sigma biLim}}}
    \; {#2}
   }

\newcommand{\dcell}[1]  
          {
					 \ar@<8pt>@{-}[d]+<-4pt,8pt> 
           \ar@<-8pt>@{-}[d]+<4pt,8pt>
           \ar@{}[d]|{#1}
          }

\newcommand{\dcellb}[1]  
          {
           \ar@<10pt>@{-}[d]+<-5pt,8pt> 
           \ar@<-10pt>@{-}[d]+<5pt,8pt>
           \ar@{}[d]|{#1}
          }

\newcommand{\deq}        
         {
          \ar@{=}[d]
         }
        
\newcommand{\dreq}      
         {
          \ar@{=}[dr]
         }

\newcommand{\dleq}      
         {
          \ar@{=}[dl]
         }

\newcommand{\dccell}[1]  
          {
           \ar@{-}[ld] 
           \ar@{-}[rd] 
           \ar@{}[d]|{#1}  
          }

\newcommand{\dcellbb}[1]  
          {
           \ar@<20pt>@{-}[d]+<-10pt,12pt> 
           \ar@<-20pt>@{-}[d]+<10pt,12pt>
           \ar@{}[d]|{#1}
          } 

\newcommand{\dl}   
          {                        
           \ar@<-2pt>@{-}[d]+<4pt,8pt>
          }

\newcommand{\dr}   
          {                        
           \ar@<2pt>@{-}[d]+<-4pt,8pt> 
          }
\newcommand{\dc}[1]   
          {                        
           \ar@{}[d]|{#1}  
          }
\newcommand{\dcr}[1]  
          {                        
           \ar@{}[dr]|{#1}  
          }

\newcommand{\uccell}[1]  
          { 
           \ar@{-}[ur] 
           \ar@{}[u]|{#1} 
           \ar@{-}[ul] 
          }

\newcommand{\uccellb}[1] 
          { 
           \ar@<-1ex>@{-}[ur] 
           \ar@{}[u]|{#1} 
           \ar@<1ex>@{-}[ul] 
          }

\newcommand{\dcellop}[1]  
          {
					 \ar@<6pt>@{-}[d]+<6pt,8pt> 
           \ar@<-6pt>@{-}[d]+<-6pt,8pt>
           \ar@{}[d]|{#1}
          }

\newcommand{\dcellopb}[1] 
          {
					 \ar@<7pt>@{-}[d]+<7pt,8pt> 
           \ar@<-7pt>@{-}[d]+<-7pt,8pt>
           \ar@{}[d]|{#1}
          }

\newcommand{\did}{\ar@2{-}[d]}

\newcommand{\op}[1]
          {
           \ar@{-}[ld] 
           \ar@{-}[rd] 
           \ar@{}[d]|{#1}  
          }

\newcommand{\cl}[1]
          { 
           \ar@{-}[ur] 
           \ar@{}[u]|{#1} 
           \ar@{-}[ul] 
          }

\newcommand{\s}{$\sigma$}

\newenvironment{acknowledgements}
  {
   \begin{abstract}}
  {\end{abstract}
   }

\begin{document}
\title{Sigma limits in 2-categories and flat pseudofunctors}
\author{Descotte M.E., Dubuc E.J., Szyld M.}

\date{\vspace{-5ex}}

\maketitle

\begin{abstract}

In this paper we introduce sigma limits (which we write $\sigma$-limits), a concept that interpolates between lax and pseudolimits: for a fixed family $\Sigma$ of arrows of a 2-category $\cc{A}$, a 
$\sigma$-cone for a $2$-functor $\cc{A} \mr{F} \cc{B}$ is a lax cone such that the structural 2-cells corresponding to the arrows of $\Sigma$ are invertible. The conical \emph{$\sigma$-limit} of $F$ is the universal $\sigma$-cone. Similary we define $\sigma$-natural transformations and weighted 
$\sigma$-limits. We consider also the case of bilimits. 
We develop the theory of $\sigma$-limits and $\sigma$-bilimits, whose importance relies on the following key fact: 
\emph{any weighted $\sigma$-limit (or $\sigma$-bilimit) can be expressed as a conical one}.
From this we obtain, in particular, a canonical expression of an arbitrary $\Cat$-valued 2-functor as a conical $\sigma$-bicolimit of representable 2-functors, for a suitable choice of $\Sigma$, which is equivalent to the well known bicoend formula.

As an application, we establish the 2-dimensional theory of flat pseudofunctors. We define a $\Cat$-valued pseudofunctor to be flat when its left bi-Kan extension along the Yoneda \mbox{$2$-functor} preserves finite weighted bilimits. We introduce a notion of $2$-filteredness of a $2$-category with respect to a class $\Sigma$, which we call \emph{$\sigma$-filtered}.  Our main result is: 
\emph{A pseudofunctor $\cc{A} \mr{} \cc{C}at$ is flat if and only if it is a $\sigma$-filtered \mbox{$\sigma$-bicolimit} of representable 2-functors.}
In particular the reader will notice the relevance of this result for the development of a theory of $2$-topoi.

\end{abstract}
 
\begin{acknowledgements}
\noindent We are grateful to the referees of AiM for their comments on this paper, and in particular for the suggestion to give proper emphasis to our results on conical $\sigma$-limits.
\end{acknowledgements}

\subsection*{Introduction}

  Size issues are in general not relevant in this paper, but we indicate the smallness assumption when it applies.
 The article is concerned with the notion of flat functor in the context of $2$-categories. Given a $2$-functor $\cc{A} \mr{P} \cc{C}at$ from a $2$-category $\cc{A}$ with values in the $2$-category $\cc{C}at$ of small \mbox{$1$-categories,} we want to define when it should be considered \emph{flat}, and prove a theorem that characterizes flatness using 
 appropriate notions of \emph{filteredness} and \emph{pro-representability}.
Recall that a $\cc{S}et$-valued functor is flat           
when its left Kan extension along the Yoneda embedding 
is left exact (this being equivalent to its discrete cofibration being a cofiltered category). This notion is considered in \cite[\S$\,$6]{K3} for
$\cc{V}$-enriched categories in general, and in particular for 
$\cc{V} = \cc{C}at$. 

We emphasize that 2-dimensional category theory is radically different from the theory of \mbox{$\cc{C}at$-enriched} categories, which, as well as the theory of $\cc{V}$-enriched categories for any $\cc{V}$, is a part of $1$-dimensional category theory.  

As is usually the case, the 
$\cc{C}at$-enriched version of flatness is too strict, and a relaxed notion is the important one. This is easily settled, but more difficult and unsolved so far is the fundamental equivalence between flatness and appropriate notions of 
filteredness for $2$-categories and pro-representability of 
$2$-functors, which is the problem solved in this article. 

Note that by 2-category, 2-functor, we mean the concepts which are sometimes referred to as strict 2-category, strict 2-functor. 
Though our original objective was to have results for $2$-functors, the notion of \emph{pseudofunctor} was imposed upon us as the correct generality in which to define flatness in the $2$-dimensional context. 
However, for the sake of simplicity in many calculations we work primarily with \mbox{$2$-functors.} We note that no generality is lost since, while we prove our main theorem (Theorem \ref{th:main}) for $2$-functors, the corresponding theorem for pseudofunctors (Theorem \ref{th:parapseudo}) follows as a corollary.

We define a pseudofunctor to be \emph{left exact} if it preserves finite weighted bilimits. 
For a pseudofunctor $\A \mr{P} \Cat$, we define the left bi-Kan extension (as already considered in \cite{PKan1}) pseudofunctor 
$\cc{H}om_p(\cc{A}^{op}, \cc{C}at) \mr{P^*} \cc{C}at$ along the Yoneda $2$-functor $\cc{A} \mr{h} \cc{H}om_p(\cc{A}^{op}, \cc{C}at)$ (namely, the bi-universal pseudonatural transformation $P \rimply P^* h$, where $\cc{H}om_p(\cc{A}^{op}, \cc{C}at)$ is the $2$-category of $2$-functors, pseudonatural transformations and modifications). 
Note that since weighted bilimits exist in $\cc{C}at$ this is actually a \emph{pointwise} bi-Kan extension. Furthermore,  
bilimits in $\cc{C}at$ can be chosen to be pseudolimits, and then it follows that when $P$ is a $2$-functor, $P^*$ can be chosen to be a $2$-functor. 
Also note that from these definitions it follows that the flatness of a $2$-functor $P$, which we define stipulating that $P^*$ is left exact, is preserved by pseudonatural equivalences, that is, equivalences in $\cc{H}om_p(\cc{A},\, \cc{C}at)$.

Let $\cc{A} \mr{P} \cc{C}at$, $\cc{A}^{op} \mr{F} \cc{C}at$ be 2-functors, consider the 2-Grothendieck construction $\cc{E}l_P \mr{\Diamond_P} \cc{A}$ and the family $\cart_P$ given by the (co)cartesian morphisms, note that we abuse the notation and consider this family both in $\cc{E}l_P$ and $\cc{E}l_P^{op}$.
While confronting the difficulty posed by the fact that there was no expression in terms of a
conical colimit indexed in $\cc{E}l_P^{op}$ for the coend of the 2-functor $\cc{A}^{op} \times \cc{A}  \xr{F \times P}  \cc{C}at$, a crucial moment
that opened the door for the intended research on the 2-dimensional concept of flatness was
the discovery of the following previously unknown fact: 

\vspace{1ex}

\emph{The category of pseudo-dicones for the \mbox{$2$-functor} 
\mbox{$\cc{A}^{op} \times \cc{A}  \xr{F \times P}  \cc{C}at$} is isomorphic to the category of 
lax cones for the $2$-functor 
\mbox{$\cc{E}l_P^{op} \mr{\Diamond_P^{op}} \cc{A}^{op} \mr{F} \cc{C}at$} such that the structural $2$-cells corresponding to $\cart_P$ are invertible.}

\vspace{1ex}

This fact led us to consider a general notion that we call \emph{sigma natural transformation}, and denote $\sigma$-natural transformation, \mbox{already} defined in \cite{GRAY}. Let $(\cc{A},\; \Sigma)$ be a pair where $\cc{A}$ is a $2$-category and $\Sigma$ a distinguished \mbox{$1$-subcategory}. A \s-natural transformation is a lax natural transformation such that the structural $2$-cells corresponding to the arrows of $\Sigma$ are invertible. 
 This notion led in turn to the notion of weighted $\sigma$-limit, which became an essential tool for our work in this paper.
We comment that although the statement in italics above follows from Proposition \ref{eq:paraintro}, it also admits
a direct proof which we encourage the interested reader to do.

It would be appropriate to say that the most transcendental basic result in this paper is Theorem \ref{colimdaconico} which establishes the following:

\vspace{1ex}

\emph{Arbitrary weighted $\sigma$-limits (or $\sigma$-colimits) can be expressed as conical ones}. 

\vspace{1ex} 

This fact rescues for 2-dimensional category theory the
classical fact of ordinary category theory which states that conical limits suffice to construct
all weighted limits. As a first application we show three statements which are the 2-categorical
analogues of the respective classical facts
of the theory of flat functors.

\vspace{1ex}

\noindent
1. We establish a $2$-categorical version of the canonical expression of $\cc{S}et$-valued functors as colimits of representable functors: Any $2$-functor $\A \mr{P} \Cat$ is equivalent to the conical \mbox{$\sigma$-colimit} of the diagram $\cc{E}l_P^{op} \mr{\Diamond_P^{op}} \cc{A}^{op} \mr{} \cc{H}om_p(\A,\Cat)$ of representable $2$-functors, where $\sigma$ is taken with respect to the class $\cart_P$ of cartesian arrows.

\vspace{1ex}

\noindent
2. We introduce a notion of $2$-filteredness for pairs 
$(\cc{A}, \; \Sigma)$ that we denote by $\sigma$-filteredness. 
When $\cc{A}$ has finite weighted bilimits and $P$ is left exact, the pair ($\cc{E}l_P^{op}$,$\,\cart_P$) is \s-filtered,  
in other words the $\sigma$-colimit in the canonical expression of $P$ is a $\sigma$-filtered $\sigma$-colimit.

\vspace{1ex}

\noindent
3. We prove a key result that establishes that a $\sigma$-filtered $\sigma$-colimit of flat $2$-functors is flat. This follows from the commutativity (up to equivalence) of $\sigma$-filtered $\sigma$-colimits with finite weighted bilimits in $\Cat$, established in \cite{Otropaper}.

\vspace{1ex}

 Let $\A \mr{P} \Cat$ be a $2$-functor. Our main result, Theorem \ref{th:main}, states that the following are equivalent:

\noindent
 (i) $\cc{E}l_P$ is $\sigma$-cofiltered with respect to the family $\cart_P$ of cocartesian arrows.

\noindent
  (ii) $P$ is equivalent to a $\sigma$-filtered $\sigma$-colimit of representable \mbox{$2$-functors in $\cc{H}om_p(\A,\Cat)$.}

\noindent
  (iii) $P$ is flat.

If $\cc{A}$ has finite weighted bilimits, these statements are also equivalent to:

\noindent
  (iv) $P$ is left exact.

\vspace{1ex}

We remark that the concept of $\sigma$-limit and the results in this paper should be relevant in an intended definition of the concept of $2$-topos. In fact, it follows that for a small $2$-category $\cc{A}$, a \emph{point} of the \emph{$2$-topos} 
$\cc{H}om_p(\A^{op},\Cat)$ of \emph{$2$-presheaves} could be appropriately defined as a 
$\cc{C}at$-valued flat $2$-functor $\A \mr{} \cc{C}at$, whose left bi-Kan extension determines a morphism of $2$-topoi 
$\cc{H}om_p(\A^{op},\Cat) \mr{} \cc{C}at$; that is, a left adjoint (i.e., having a right adjoint) left exact $2$-functor. As A. Joyal pointed to us, a $2$-topos could be defined as a left exact \mbox{$2$-localization} of a $2$-category of $2$-presheaves.

\subsection* {Organization of the paper} 

In Section \ref{prelim} we fix notation and terminology. 
Through Sections \ref{sigmal} and \ref{sigma2f} we fix an arbitrary pair 
$(\cc{A}, \; \Sigma)$ with $\Sigma$ a $1$-subcategory containing all the objects of a $2$-category $\cc{A}$.

\vspace{1ex}

In Section \ref{sigmal} we develop the theory of \s-limits. In \S\,\ref{sub:sigmanatural} we define \mbox{\s-natural} transformations between $2$-functors 
\mbox{$\cc{A} \mr{} \cc{B}$} following 
\mbox{\cite[\S$\,$I,2 p.13,14]{GRAY}.} These are lax natural transformations where the $2$-cells associated to the arrows in $\Sigma$ are invertible. We denote the so determined $2$-category  $\cc{H}om_\sigma^{\Sigma}(\cc{A}, \, \cc{B})$, and whenever possible we will omit $\Sigma$ from the notation. In this way we have a chain of inclusions of categories with the same objects:
$$\cc{H}om_s(\A,\B) \hookrightarrow \cc{H}om_p(\A,\B) \stackrel{(1)}{\hookrightarrow} \cc{H}om_{\sigma}(\A,\B) \stackrel{(2)}{\hookrightarrow} \cc{H}om_{\ell}(\A,\B)$$
\noindent where the sub indexes  $s$, $p$, \s, $\ell$  indicate strict natural (i.e. $2$-natural), pseudonatural, \mbox{\s-natural} and lax natural respectively. When $\Sigma$ is the whole underlying category of $\A$, $(1)$ above is an equality, and when $\Sigma$ consists only of the identities $(2)$ is so. This allows for a unified treatment of many results known for pseudo and  lax natural transformations. 

Each choice of a subindex $s$, $p$, \s, $\ell$ gives rise to a notion of weighted limit that we study in 
\S\,\ref{sub:epslimits}. 
Note that the three cases $s$, $p$, $\ell$ are considered in \cite{K2}, but the general concept of $\sigma$-limit for an arbitrary $1$-subcategory $\Sigma$ is an essential tool to work with the notion of flat $2$-functor, and we use in this paper $\sigma$-limits that are neither lax nor pseudolimits.

\emph{Notation: In order to avoid repeating statements and, more important, to develop unified proofs whenever possible, we will use a letter $\eps$, that can stand for both \textquotedblleft$s$\textquotedblright\ and \textquotedblleft$\sigma$\textquotedblright, thus also for \textquotedblleft$p$\textquotedblright \ and \textquotedblleft$\ell$\textquotedblright.}

\emph{Warning: we use \emph{limit} to refer to a general weighted limit, and \emph{conical limit} to a classical limit (i.e., when the weight is the constant $2$-functor with value $1$)}.

In \S\,\ref{sub:epsends} we consider for arbitrary $\eps$ the corresponding notion of $\eps$-$end$ and $\eps$-$coend$, and establish the $\eps$-end formula for the category of $\eps$-natural transformations between $2$-functors. We consider also \emph{tensors} and \emph{cotensors}, and prove the constructions of weighted $\eps$-limits in terms of $\eps$-ends and cotensors. 

In \S\,\ref{sub:conicalsigma} we study explicitly conical 
$\sigma$-limits and $\sigma$-colimits, and show, modifying an argument of Street \cite{S}, the fundamental property of 
$\sigma$-limits that we mention in the introduction. We choose to establish it for colimits:  \emph{arbitrary $\sigma$-colimits can be expressed as conical 
$\sigma$-colimits}.
We establish then the canonical expression of a $\Cat$-valued $2$-functor as a conical 
\mbox{$\sigma$-bicolimit} of representable $2$-functors. We finish this subsection adapting Gray's construction of \mbox{$\sigma$-colimits} in 
$\Cat$ to fit our context, which is a result that we will need later.

In \S\ \!\ref{sub:pointwise} we analyze the computation of weighted $\eps$-limits in $2$-functor categories, and establish a general theorem about pointwise computation. This is an essential theorem in the theory of limits, which is used everywhere. In particular, we use it in 
\S\ \!\ref{sub:interchange} in order to 
prove properties of interchange of $\eps$-limits 
and $\eps$-colimits.

\vspace{1ex}

In Section \ref{sigma2f} we introduce and develop the notion of $2$-filteredness for pairs $(\cc{A}, \; \Sigma)$, which we refer to by saying that $\cc{A}$ is \emph{$\sigma$-filtered} (with respect to $\Sigma$).
In \S$\,$\ref{sub:sigmafiltered} we state the basic definition, which is a generalization of Kennison's three axioms in his definition of \emph{bifiltered} $2$-category \cite{K}, thus it also generalizes the equivalent 
Dubuc-Street notion of \emph{$2$-filtered} \mbox{$2$-category} \cite{DS}. Their notion corresponds to \s-filteredness when $\Sigma$ consists of all the arrows of $\cc{A}$. 
We consider particular finite diagrams such that their $\sigma$-cones suffice for $\sigma$-filteredness, and show that these $\sigma$-cones correspond (up to equivalence) to the cones of some particular finite weighted bilimits.
In \S$\,$\ref{sub:exact} we consider the pair ($\cc{E}l_P$,$\cart_P$) as mentioned in the introduction and 
we prove that the 
$2$-functor \mbox{$\cc{E}l_P \mr{\Diamond_P} \cc{A}$} creates any conical \s-bilimit which exists in $\cc{A}$ and is preserved by $P$ 
(this is a $2$-dimensional version of a known $1$-dimensional result, see \mbox{\cite[Proposition 4.87]{K1}).} 
From this result, together with the equivalence between cones mentioned above, 
it follows that if $\cc{A}$ has finite weighted bilimits and $P$ is left exact, then the pair ($\cc{E}l_P^{op}$, $\cart_P$) is 
$\sigma$-filtered. 
Interestingly enough, finite conical bilimits in $\cc{A}$ do not suffice for this result.
In \S\,\ref{sub:cofinal} we consider \emph{\s-cofinal} $2$-functors and establish some of the usual properties of cofinality that we will use in the proof of our main theorem in Section~\ref{flat}. These properties allow us to show that the canonical $2$-functor $\cc{E}l_P^{op} \mr{} \cc{E}l_{L}^{op}$, where $L$ is a left bi-Kan extension of $P$, is $\sigma$-cofinal in the case considered in the theorem.

\vspace{1ex}

In Section \ref{flat} we consider flat pseudofunctors and we prove our main theorem. In \S\,\ref{sub:flatpseudo} we 
define the \emph{bi-Kan extension} of a pseudofunctor following 
\cite{PKan1}. It is defined by the usual representation that defines Kan extensions suitably relaxed. We focus on the \emph{pointwise} case which holds when the target $2$-category has all weighted bilimits, and prove some basic results on flat pseudofunctors, analogous (but independent since the two notions of flatness are different) to the ones that can be found for a general base category $\cc{V}$ in \mbox{\cite[\S$\,$6]{K3}.} 
In \S$\,$\ref{sub:maintheorems} we state and prove the results mentioned in the introduction, in particular our main theorem (Theorem \ref{th:main}), and in Appendix \ref{appendix} we generalize them to the case of pseudofunctors.

\section{Preliminaries}\label{prelim}

\subsection{Basic terminology} \label{sub:terminology}

Since terminology regarding $2$-dimensional category theory varies in the literature, we list here some definitions and basic results as we will use them in this paper. 

\begin{enumerate}

\item We refer the reader to \cite{KS} for basic notions on 2-categories. 
Size issues are not relevant to us
here, when it is not clear from the context we indicate the smallness condition if it applies.

\item In any 2-category, we use $\circ$ to denote vertical composition and juxtaposition to denote horizontal composition. We consider juxtaposition more binding than ``$\circ$'',
thus $\alpha \beta \circ \gamma $ means $(\alpha \beta ) \circ \gamma$. 
We will abuse notation by writing $f$ instead of $id_f$ for arrows $f$ when there is no risk of confusion.

\item Given any arrow or 2-cell ``$x$'', we use ``$x^*$'', ``$x_*$'' to denote precomposition, postcomposition with ``$x$'' respectively.

\item By $\Cat$ we denote the $2$-category of (small) categories, with functors as morphisms and natural transformations as 2-cells. 

\item For a $2$-category $\A$ and objects $A, B \in \A$, we use the notation $\A(A,B)$ to denote the category whose objects are the morphisms between $A$ and $B$ and whose arrows are the 2-cells between those morphisms.

\item We use $\cong$ to denote isomorphisms and $\approx$ to denote equivalences in a 2-category.

\item\label{it:ff} A $2$-functor $F:\cc{A}\mr{} \cc{B}$ is said to be \mbox{\emph{pseudo-fully-faithful}} if for each $A,\ B\in \cc{A}$, $\cc{A}(A,B) \xr{F_{A,B}} \cc{B}(FA,FB)$ is an equivalence of categories, \emph{2-fully-faithful} if each $F_{A,B}$ is an isomorphism and \mbox{\emph{locally-fully-faithful}} if each $F_{A,B}$ is full and faithful.

\item\label{def:op} For a $2$-category $\A$, $\A^{op}$ denotes the $2$-category with the same objects as $\A$ but with $\A^{op}(A,B)=\A(B,A)$, i.e. we reverse the $1$-cells but not the $2$-cells. We use the notation $\overline{B} \mr{f} \overline{A}$ for the arrow in $\A^{op}$ that corresponds to the arrow $A \mr{f} B$, in $\A$. $2$-cells keep their names.

\item \label{item:co} For a $2$-category $\A$, $\A^{co}$ denotes the $2$-category with the same objects and arrows as $\A$, but with 
$\A^{co}(A,B)=\A(A,B)^{op}$, i.e. we reverse the $2$-cells but not the $1$-cells. 

\item \label{item:isoop} The $2$-category $\Cat$ has a duality $2$-functor $\Cat^{co} \mr{D} \Cat$ that maps each category $C$ to its dual $C^{op}$. Clearly $D$ is an isomorphism of $2$-categories and it is its own inverse.

\item\label{laxoplaxnatural} A lax natural transformation between $2$-functors $\A \mrpair{F}{G} \B$ is a family of morphisms and $2$-cells of $\B$, $\{FA \mr{\theta_A} GA\}_{A \in \A}$,  $\{Gf \theta_A \Mr{\theta_f} \theta_B Ff\}_{A\mr{f}B \in \A}$ satisfying the following equations:

\noindent
\begin{tabular}{rll}
 {\bf LN0}. & For all $A\in \A$, & ${\theta}_{id_A} = \theta_A$. \\
 {\bf LN1}. & For all $A \mr{f} B \mr{g} C \in \A$, & $\theta_{gf} =\theta_g Ff \circ Gg \theta_f$. \\
 {\bf LN2}. & For all $A \cellrd{f}{\gamma}{g} B \in \A$, & $\theta_B F\gamma  \circ \theta_f = \theta_g \circ G\gamma \theta_A$. 
\end{tabular}

An op-lax natural transformation is defined analogously but the structural $2$-cells $\theta_f$ are reversed, i.e. $\theta_B Ff \Mr{\theta_f} Gf \theta_A$.

A modification $\theta \mr{\rho} \theta'$ between lax natural transformations is a family of $2$-cells of $\B$ $\{\theta_A \Mr{\rho_A} \theta'_A\}_{A \in \A}$ such that:

\noindent
\begin{tabular}{rll}
 {\bf LNM}. & For all $A \mr{f} B \in \A$, & $\quad \quad \; \theta'_f \circ Gf \rho_A = \rho_B Ff \circ \theta_f$. 
\end{tabular}

In this way we have a $2$-category $\cc{H}om_{\ell}(\A,\B)$, with arrows the lax natural transformations, and similarly $\cc{H}om_{op \ell }(\A,\B)$.

A pseudonatural transformation is a lax natural transformation where all the $2$-cells $\theta_f$ are invertible, they are the arrows of a $2$-category $\cc{H}om_{p}(\A,\B)$. A strict, or $2$-natural transformation is a lax natural transformation where all the $2$-cells $\theta_f$ are identities, they are the arrows of a $2$-category $\cc{H}om_s(\A,\B)$. We have locally-fully-faithful inclusions
  \begin{equation} \label{eq:inclusionesHomssinsigma}
 \cc{H}om_s(\A,\B) \hookrightarrow \cc{H}om_p(\A,\B) \hookrightarrow \cc{H}om_{\ell}(\A,\B)  
  \end{equation} 
and similarly for $\cc{H}om_{op \ell }(\A,\B)$. 
A pseudonatural equivalence, or pseudo-equivalence for short, is a pseudonatural transformation such that every $\theta_A$ is an equivalence in $\B$. This amounts to $\theta$ being an equivalence in $\cc{H}om_p(\A,\B)$.

\item \label{cartesianHomp}
There is a bijective correspondence between  $2$-functors, where $\gamma$ is either $s$, $p$ or $\ell$:

\vspace{-2ex}
 
$$
\xymatrix@R=0.5ex@C=2ex
   {
     { } & \B  \ar[rrr]^(0.35)F  & & & \cc{H}om_\gamma(\A,\, \C)
    \\
    { } \ar@{-}[rrrrr] & & & & & { }
    \\
     { } & \A  \ar[rrr]^(0.35)G  & & &  \cc{H}om_{op\gamma}(\B,\, \C)
   }
$$

\vspace{-2ex}
 
This correspondence is given by the formulas, for $A \cellrd{f}{\eta}{f'} A'$, 
$B \cellrd{g}{\theta}{g'} B'$: 

$FB(A) = GA(B)$, \mbox{$(Fg)_A = GA(g)$,} $FB(f) = (Gf)_B$, $(Fg)_f = (Gf)_g$, $(F\eta)_A = GA(\eta)$, $FB(\theta) = (G\theta)_B$.  All the verifications are straightforward. 

The expression $H(A, \, B) = FB(A) = GA(B)$ does not determine a $2$-functor of two 
variables, its structure has been studied in \cite[I, 4.1.]{GRAY} under the name of \emph{quasifunctor}. 

\item\label{laxdinat} A lax dinatural transformation $\theta$ between $2$-functors $\A^{op} \times \A \mrpair{F}{G} \B$ is a family of morphisms and $2$-cells of $\B$, $\{F(A,A) \mr{\theta_A} G(A,A)\}_{A \in \A}$,  
\mbox{$\{G(id,f) \theta_A F(f,id) \Mr{\theta_f} G(f,id) \theta_B F(id,f)\}_{A\mr{f}B \in \A}$} satisfying the following equations:

\noindent
\begin{tabular}{rll}
{\bf LD0.} & \!\!\!\!\!\! For all $A \in \A$, & \!\!\! ${\theta}_{id_A} = \theta_A$. \\
{\bf LD1.} & \!\!\!\!\!\! For all $A \mr{f} B \mr{g} C \in \A$, & \!\!\! $\theta_{gf} = G(f,id) \theta_g F(id,f) \circ G(id,g) \theta_f F(g,id)$. \\
{\bf LD2.} & \!\!\!\!\!\! For all $A \cellrd{f}{\gamma}{g} B \in \A$, & \!\!\! $G(\gamma,id) \theta_B F(id,\gamma) \circ \theta_f = \theta_g \circ G(id,\gamma) \theta_A F(\gamma,id)$. 
\end{tabular}

A morphism $\rho$ between two lax dinatural transformations $\theta,\theta'$ is a family of $2$-cells of $\B$, $\{\theta_A \Mr{\rho_A} \theta'_A\}_{A \in \A}$ such that:

\noindent
\begin{tabular}{rll}
 {\bf LDM.} & For all $A \mr{f} B \in \A$, & $\theta'_f \circ G(id,f) \rho_A F(f,id) = G(f,id) \rho_B F(id,f) \circ \theta_f$. 
\end{tabular}

Note that if a pair of $2$-functors $\A \mrpair{F}{G} \B$ are considered as $2$-functors $\A^{op} \times \A \mrpair{\widetilde{F}}{\widetilde{G}} \B$ constant in the first variable, lax dinatural transformations from $\widetilde{F}$ to $\widetilde{G}$ correspond to lax natural transformations from $F$ to $G$, and similarly for their morphisms.

\item \label{item:coco} The construction of item \ref{item:co} defines an isomorphism $\cc{H}om_{\ell}(\A,\B) \mr{(-)^{co}} \cc{H}om_{op \ell }(\A^{co},\B^{co})$.

\item \label{item:pirulo} Combining the previous item with item \ref{item:isoop}, we have an isomorphism of $2$-categories $\cc{H}om_{\ell}(\A,\Cat) \mr{(-)^{co}} \cc{H}om_{op \ell }(\A^{co},\Cat^{co}) \mr{D_*} \cc{H}om_{op \ell }(\A^{co},\Cat)$
that maps a $2$-functor $\A \mr{P} \Cat$ to a $2$-functor that we will denote by $P^{d}$, $P^{d} = DP^{co}$.

\item \label{item:conicalweighted} For a $2$-functor $\A \mr{F} \B$, and an object $E \in \B$ we have isomorphisms of categories 
$$\cc{H}om_{\ell}(\A,\Cat)(k_1, \B(E,F-)) \cong \cc{H}om_{\ell}(\A,\B)(k_E, F)$$ 
$$\cc{H}om_{\ell}(\A^{op},\Cat)(k_1, \B(F-,E)) \cong \cc{H}om_{op \ell }(\A,\B)(F,k_E)$$ 
\noindent where $k_1$ and $k_E$ denote the $2$-functors constant at $1 = \{*\}$, and $E$ respectively.

For $\theta$ in the left side and $\eta$ in the right side, both isomorphisms are given by the formulas $\eta_A=\theta_A(*)$ for $A\in \A$, $\eta_f=(\theta_f)_{*}$ for $A\mr{f} B \in \A$.
\end{enumerate}

\subsection{The $2$-category of elements} \label{sin:prelimgrothconstructions}

 We will make extensive use of the $2$-category of elements $\cc{E}l_P$ of a $\Cat$-valued $2$-functor $P$. 
 $\cc{E}l_P$ can be defined as a particular instance of a \emph{lax comma $2$-category} (\cite[\S$\,$1.4]{Bird}, \cite[\S$\,$I,2.5]{GRAY}), 
 $\cc{E}l_P = [k_1,P]$, 
 and therefore has the universal property of Proposition \ref{prop:pudegammaP} below.

 \begin{\de}\label{def:gammaP}
 Let $\A \mr{P} \Cat$ be a $2$-functor. $\cc{E}l_P$ can be described as follows:
 
 \begin{enumerate}
  \item Objects: Pairs $(x,A)$ with $A \in \A$ and $x \in PA$
  \item Morphisms: A morphism between $(x,A)$ and $(y,B)$ is a pair $(f,\varphi)$ with $A \mr{f} B \in \A$ and $Pf(x) \mr{\varphi} y$
  \item 2-cells: A 2-cell between $(f,\varphi)$ and $(g,\psi)$ (from $(x,A)$ to $(y,B)$) is given by a 2-cell $A\cellrd{f}{\theta}{g}B \in \A$ such that the following diagram commutes in $PB$:
  
  $$\xymatrix{Pf(x) \ar[d]_{(P\theta)_x}  \ar[r]^{\varphi} & y \\
              Pg(x) \ar[ru]_{\psi}   }$$
              
 \item Compositions in this $2$-category are defined as follows: for composable arrows 
  $(f,\varphi)$ and $(g,\psi)$ 
 we have $(g,\psi)(f,\varphi) = (gf,\psi Pg(\varphi))$, and 
both horizontal and vertical composition of $2$-cells are computed in $\A$.
 \end{enumerate}

 We consider the $1$-subcategory $\cart_P$ of $\cc{E}l_P$ whose arrows are $(f,\varphi)$ with $\varphi$ an isomorphism. 
 \end{\de}

\begin{remark} \label{rem:diamondP}
 We note that the canonical projection $\cc{E}l_P \mr{\Diamond_P} \A$ is the opfibration (in the sense of \cite[\S$\,$I,2.5 p.30]{GRAY}) associated to $P$, and the arrows of $\cart_P$ are the cocartesian morphisms of $\cc{E}l_P$. 
\end{remark} 

\begin{\prop} [{\cite[\S$\,$I,2.5 p.29]{GRAY}, \cite[Proposition 1.11]{Bird}}] \label{prop:pudegammaP}
 The following diagram expresses the fact that (together with $\Diamond_P$ and the lax natural transformation $\alpha$ defined by
 $\alpha_{(x,A)} = x$, $\alpha_{(f,\varphi)} = \varphi$), $\cc{E}l_P$ is the lax pull-back of $P$ along the $2$-functor $1 \mr{k_1} \Cat$.
 
 $\hbox{For each } 2 \hbox{-functor } \cc{Z} \mr{F} \A, \hbox{ and each lax natural transformation } k_1 \mr{\theta} PF,$ 
  
 $$\vcenter{\xymatrix{\cc{Z} \ar@/_2ex/[rdd] \ar@/^2ex/[rrd]^F \ar@{.>}[rd]^{\exists ! T} \\
 & \cc{E}l_P \ar[d] \ar[r]^{\Diamond_P} \ar@{}[rd]|{\stackrel{\alpha}{\Rightarrow}} & \A \ar[d]^{P} \\
 & 1 \ar[r]_{k_1} & \Cat}}
 \hbox{ such that } \Diamond_P T = F, \alpha T = \theta.
 $$
 
 The formulas behind this correspondence are, for $Z \cellrd{r}{\beta}{s} W$ in $\cc{Z}$, $T(Z) = (\theta_Z, F(Z))$, $T(r) = (F(r),\theta_r)$, $T(\beta)=F(\beta)$.
 There is also a $2$-categorical part of this universal property that we omit since we will not use it, the reader may consult \cite[Proposition 1.11]{Bird}.  \qed
\end{\prop}

\begin{remark} \label{rem:oplaxdense}
  It is well-known (see \cite[p.180]{S}, or see \cite[Proposition 1.14]{Bird} for a proof) that the projection $\cc{E}l_P \mr{\Diamond_P} \A$ is \textit{lax dense}, in the sense that for each $\A \mr{Q} \Cat$ the pasting composition with $\alpha$ yields an isomorphism of categories 
 
 $$\cc{H}om_{\ell}(\A,\Cat)(P,Q) \cong \cc{H}om_{\ell}(\cc{E}l_P,\Cat)(k_1, Q \Diamond_P).$$
 
 We make explicit the formulas defining this correspondence on objects:
 
 \begin{center}
 \begin{tabular}{ccc}
 $P \Mr{\eta} Q$ lax natural && $k_1 \Mr{\theta} Q \Diamond_P$ lax natural \\
 $\xymatrix{PA \ar[r]^{\eta_A} \ar[d]_{P(f)} \ar@{}[dr]|{\Downarrow \eta_f} & QA \ar[d]^{Q(f)} \\ PB \ar[r]_{\eta_B} & QB}$ &
  & 
 $\xymatrix{1 \ar[r]^{\theta_{(x,A)}} \ar@/_2ex/[rd]_{\theta_{(y,B)}} & QA \ar[d]^{Q(f)}_<<<<{\Downarrow \theta_{(f,\varphi)}} \\ & QB}  $ \\
 $(\eta_f)_x = \theta_{(f,id)}$ & $\eta_A(x) = \theta_{(x,A)}$ & $\theta_{(f,\varphi)} = \eta_B(\varphi) (\eta_f)_x$
 \end{tabular}
 \end{center}
\end{remark}

\begin{sinnadastandard} \label{sin:Tsubeta}
Combining Proposition \ref{prop:pudegammaP} and Remark \ref{rem:oplaxdense} we have, for each lax natural transformation $P \Mr{\eta} Q$ between $\Cat$-valued $2$-functors, an induced $2$-functor $\cc{E}l_P \mr{T_{\eta}} \cc{E}l_Q$ given by the formulas 
$$ 
T_{\eta}(x,A) = (\eta_A(x),A), \quad T_{\eta}(f,\varphi) = (f, \eta_B(\varphi) (\eta_f)_x), \quad T_\eta(\theta)=\theta.
$$
We note (see \cite[Theorem 1.15]{Bird}), although we will not need this result, that this assignment actually defines a $2$-fully-faithful $2$-functor $\cc{H}om_{\ell}(\A,\Cat) \mr{} (2\hbox{-}\Cat/\A)$.
\end{sinnadastandard}

\begin{sinnadastandard} \label{sin:TsubH}
 Consider now $2$-functors $\A \mr{H} \B \mr{P} \Cat$. By the pasting lemma for lax pull-backs, we may construct the lax pull-back defining $\cc{E}l_{PH}$ by pasting a (strict) $2$-pull-back to the lax pull-back defining $\cc{E}l_P$ as in the diagram below:
 
 $$\vcenter{\xymatrix{ \cc{E}l_{PH} \ar[d]_{T_H} \ar[r]^{\Diamond_{PH}} & \A \ar[d]^H \\
 \cc{E}l_P \ar[d] \ar[r]^{\Diamond_P} \ar@{}[rd]|{\stackrel{\alpha_P}{\Rightarrow}} & \B \ar[d]^{P} \\
 1 \ar[r]_{k_1} & \Cat}}$$
 
 \noindent Then we have an induced $2$-functor $\cc{E}l_{PH} \mr{T_H} \cc{E}l_P$ that is given by the formulas 
 $$
 T_{H}(x,A) = (x,HA), \quad T_{H}(f,\varphi) = (Hf,\varphi), \quad T_H(\theta)=H\theta.
$$
\end{sinnadastandard}

\section{$\sigma$-limits}\label{sigmal}

We fix throughout this section a $1$-subcategory $\Sigma$ of a $2$-category $\A$ which contains all the objects (this is often called a {\em wide} subcategory). We introduce a new class of natural transformations that we call
\emph{sigma natural}, and denote \emph{$\sigma$-natural.} We introduce the use of a symbol $\sigma$ accompanying a concept, it is convenient to think that $\sigma$ means that the concept is to be taken ``relative to the arrows of $\Sigma$''. 
Whenever possible, we will omit $\Sigma$ from the notation.

\subsection{\s-natural transformations} \label{sub:sigmanatural}

\begin{\de} \label{def:cnatural}
 Given $2$-functors $\A \mrpair{F}{G} \B$, a $\sigma$-natural transformation $F \Mr{\theta} G$ (with respect to $\Sigma$) is a lax natural transformation such that, if $A \mr{f} A'$ is in $\Sigma$, the structural \mbox{$2$-cell} 
 $\theta_f$ (see \S$\,$\ref{sub:terminology}, item \ref{laxoplaxnatural})
   is invertible. There is a $2$-category $\cc{H}om_{\sigma}^{\Sigma}(\A,\B)$ with objects the $2$-functors from $\A$ to $\B$, whose arrows are the $\sigma$-natural transformations and whose $2$-cells are 
   all the modifications between them. We have locally-fully-faithful inclusions (see \eqref{eq:inclusionesHomssinsigma})
  
  \begin{equation} \label{eq:inclusionesHoms}
 \cc{H}om_s(\A,\B) \hookrightarrow \cc{H}om_p(\A,\B) \stackrel{(1)}{\hookrightarrow} \cc{H}om_{\sigma}(\A,\B) \stackrel{(2)}{\hookrightarrow} \cc{H}om_{\ell}(\A,\B).  
  \end{equation}
  
  \noindent Note that if $\Sigma'$ is another $1$-subcategory of $\A$ and $\Sigma \subseteq \Sigma'$ then 
   $\cc{H}om_{\sigma}^{\Sigma'}(\A,\B) \hookrightarrow \cc{H}om_{\sigma}^{\Sigma}(\A,\B)$.
 \end{\de}

We recall that $\sigma$-natural transformations were already considered by J. W. Gray in \mbox{\cite[\S$\,$I,2 p.13,14]{GRAY}.} What we denote by $\cc{H}om_\sigma^{\Sigma}(\A,\B)$ is, in Gray's notation, $Fun(\A,\Sigma;\B,iso \B)$.

\begin{remark} \label{rem:cmataplyepsilon}
 Consider a $2$-category $\A$, its underlying category $\A_0$ and the $1$-subcategory $\A_{id}$ consisting only of the identities.  
 Then, in \eqref{eq:inclusionesHoms}, $(1)$ is an equality if $\Sigma = \A_0$, and $(2)$ is an equality if $\Sigma = \A_{id}$.
 \cqd 
\end{remark} 
Observe that the items 
\ref{item:coco}, \ref{item:pirulo} and \ref{item:conicalweighted} in \S$\,$\ref{sub:terminology} hold with the same proof for general $\sigma$  and $op\sigma$-natural transformations, the latter being defined in an evident way. \cqd  
\begin{notation} \label{not:eps}
  Even though in this paper we will work mainly with $\sigma$-limits, in order to avoid repeating statements that hold for the $\sigma$-case and the strict $s$-case, we will use a letter $\eps$, that can stand for both $s$ and $\sigma$ (then also by Remark \ref{rem:cmataplyepsilon} for $p$ and $\ell$). This allows for a unified treatment of many results which are known for strict, pseudo and lax natural transformations.
\end{notation}

\subsection{$\eps$-limits} \label{sub:epslimits}

\begin{\de} \label{de:NE} 
Given $2$-functors $\A \mr{W} \Cat$, $\A \mr{F} \B$, and $E$ an object of $\B$, we denote 
$Cones_{\eps}^W(E,F) = \cc{H}om_{\eps}(\A,\Cat)(W,\B(E,F-))$. This is the category of $w$-$\eps$-cones (with respect to the weight $W$) for $F$ with vertex $E$. 
For a $w$-$\eps$-cone $\xi$ with vertex $E$, 
$\xymatrix@C=3ex{W \ar@{=>}[r]^>>>>{\xi} & \B(E,F-)}$,  
we have a functor $\theta_B = \xi^*$ given by precomposition with $\xi$:

\vspace{-1ex}

\begin{equation} \label{eq:ssss}
 \B(B, E ) \mr{\theta_B} \cc{H}om_{\eps}(\A,\Cat)(W,\B(B,F-))
 \end{equation}
 
 \vspace{-2ex}
 
$$ B \cellrd{f}{\alpha}{g} E \quad \longmapsto \quad W \Mr{\xi} \B(E,F-) \cellrd{f^*}{\alpha^*}{g^*} \B(B,F-) $$

\noindent
The $\eps$-limit of $F$ weighted by $W$, denoted $\{W,F\}_\eps$ or more precisely $(\{W,F\}_\eps, \xi)$, is a $w$-$\eps$-cone $\xi$ 
with vertex $E = \{W,F\}_\eps$, universal in the sense that 
$\theta_B = \xi^*$ in 
\eqref{eq:ssss} is an isomorphism.

As usual, an equivalent formulation of the universal property is that there is a representation $\theta_B$ natural in the variable $B$ (as in \eqref{eq:ssss}), 
and $\xi$ is recovered setting $B = E$, $\xi = \theta_E(id_E)$.
\end{\de}
It is convenient to give an explicit definition of the dual concept, in the notation of Definition \ref{de:NE}:
\begin{definition} \label{sub:epscolimits}
$\eps$-colimits $W \otimes_{\eps} F$ in $\B$ are the corresponding limits in $\B^{op}$; for $\A^{op} \mr{W} \Cat$, $\A \mr{F} \B$ 
we denote $Cones_{\eps}^W(F,E) = \cc{H}om_{\eps}(\A^{op},\Cat)(W,\B(F-,E))$ and refer to the objects of this category also as $w$-$\eps$-cones, as it is clear from the context which $w$-$\eps$-cones we are referring to. 
The $\eps$-colimit of $F$ weighted by $W$, denoted \mbox{$W \otimes_\eps F$} or more precisely $(W \otimes_\eps F, \nu)$, is a $w$-$\eps$-cone $\nu$ with vertex $E = W \otimes_\eps F$, 
$\xymatrix@C=3ex{W \ar@{=>}[r]^>>>>{\nu} & \B(F-,B)}$ universal in the sense that the functor $\theta_B = \nu^*$ given by precomposing with $\nu$, 
  \begin{equation} \label{eq:defcpseudocolimconnu}
  \B(E, B) \mr{\theta_B} \cc{H}om_{\eps}(\A^{op},\Cat)(W,\B(F-,B))
 \end{equation}
 \noindent is an isomorphism.
\end{definition}
\begin{remark} \label{rem:sinnombre}
 Considering $\eps = s$, we recover the notion of strict weighted limit (\cite[\S$\,$2]{K2}).
   Considering $\eps = \sigma$, $\Sigma = \A_0$ and $\Sigma = \A_{id}$, we recover the notions of weighted lax and pseudolimits (\cite[\S$\,$5]{K2}). In spite of this notation, the reader should be aware that $s$-limits are not $\sigma$-limits, as it is the case for weighted lax and pseudolimits. 
    
  The general concept of $\sigma$-limit for an arbitrary $1$-subcategory $\Sigma$ is an essential tool to work with the notion of flat $2$-functor, and we will consider in this paper $\sigma$-limits that are neither lax nor pseudolimits.
\end{remark}

\begin{remark} \label{rem:4en1}
 We also consider (but omit to write explicitly) analogous statements for op$\sigma$-natural transformations, thus defining op$\sigma$-limits. 
 Recall \S$\,$\ref{sub:terminology}, item \ref{item:pirulo}. Every $\sigma$-limit in $\B$ is an op$\sigma$-limit in $\B^{co}$, and vice versa. If $\A \mr{W} \Cat$, $\A \mr{F} \B$, then \mbox{$\{W,F\}_{\sigma} = \{W^{d},F^{co}\}_{op \sigma }$.} See \cite[Proposition 1.5]{Bird} for a proof for lax natural transformations, that can be easily adapted to a general $\sigma$.
 
 Therefore one can think there is only one main or ``primitive'' notion between the four possible choices in (op)$\sigma$-(co)limits, and the other three can be obtained from that one. Then, as it is usual in the literature, we can state and prove general results for $\sigma$-limits, and use them for any of the four choices mentioned above.
\end{remark}

\begin{remark} \label{rem:defbi}
If the universal property in Definition \ref{de:NE} is taken in the weak sense (equivalence instead of isomorphism) we have the notion of $\sigma$-bilimit ${}_{bi} \! \{W,F\}_{\sigma}$ 
(and $\sigma$-bicolimit $W \;{}_{bi} \! \otimes_{\sigma} F$). 
Clearly, any $\sigma$-limit is in particular a $\sigma$-bilimit. 
Note however that the defining universal properties characterize $\sigma$-bilimits up to equivalence and $\sigma$-limits up to isomorphism. We abuse nevertheless the language by referring to ``the'' $\sigma$-bilimit, or ``the'' $\sigma$-limit, and use equalities to express that a certain object satisfies the corresponding universal property, but it is important to be aware that for a given data, both a $\sigma$-limit and a $\sigma$-bilimit 
may be constructed independently, and they will be equivalent objects, but \emph{not necessarily isomorphic}.
\end{remark}

As in \cite[(2.5), (5.5)]{K2} (see Remark \ref{rem:limitsenCat} below for a proof) we have the basic result:
\begin{proposition} \label{cathas}
The 2-category $\cc{C}at$ has all (small) weighted $\varepsilon$-limits. In fact, given \hbox{$\cc{A} \mr{W} \cc{C}at$, $\cc{A} \mr{P} \cc{C}at$,}
$\{W, P\}_\varepsilon = \cc{H}om_\varepsilon(\cc{A}, \cc{C}at)(W, P)$.
\cqd
\end{proposition}
As an immediate corollary, it follows that representable $2$-functors preserve weighted \mbox{$\varepsilon$-limits.} That is, they ``come out of the second variable''. More precisely:

\begin{corollary} \label{comeout}
 Let $\A \mr{W} \Cat$, $\A \mr{F} \B$ be $2$-functors, then we have the following isomorphism (equivalence), $2$-natural in the variable $B$:
$$
\cc{B}(B, \{W, F\}_\varepsilon ) \mr{\cong} 
                                 \{W, \cc{B}(B, F-)\}_\varepsilon
,\;\;\;\;
\cc{B}(B, {}_{bi} \!\{W, F\}_\varepsilon ) \mr{\approx} 
                       {}_{bi} \!\{W, \cc{B}(B, F-)\}_\varepsilon
$$
\end{corollary}
\begin{proof}
Consider $P = \B(B,F-)$ in Proposition \ref{cathas} and the Definition \ref{de:NE} of $\eps$-limit. The case of $\eps$-bilimits is analogous.
\end{proof}

It is well known (\cite[(3.11)]{K1}) that weighted strict limits behave functorially both in the weight and the argument. 
Here we establish the fact that $\eps$-limits (recall that $\eps$ stands for $\sigma$ or $s$) behave functorially respect to any natural transformation \emph{stronger} than $\eps$-natural, more precisely:

\begin{notation} \label{epsilonnotation}
 Let $\A$ be any $2$-category. Consider the set $\Lbl_\A$ consisting of the label $s$ and one label $\sigma^\Sigma$ for each \mbox{$1$-subcategory} $\Sigma$ of $\A$. Note that in particular we have labels that we denote $p = \sigma^{\A_{id}}$, $\ell = \sigma^{\A_0}$ (see Remark \ref{rem:cmataplyepsilon}). Consider the order in $\Lbl_\A$ induced by the inclusions in \eqref{eq:inclusionesHoms}, that is $s \leq \sigma^\Sigma$ for every $\Sigma$, and $\sigma^{\Sigma'} \leq \sigma^\Sigma$ if $\Sigma \subseteq \Sigma'$.
 
 Note that if we are considering only one $1$-subcategory $\Sigma$, and omit it from the notation, we have $s \leq p \leq \sigma \leq \ell$ (cf. \eqref{eq:inclusionesHoms}).
\end{notation}

\begin{remark} \label{rem:limitsfuntorial}
 
  Let $\alpha,\beta \in \Lbl_\A$, if $\alpha \leq \beta$ then weighted $\beta$-limits behave functorially respect to $\alpha$-natural transformations. That is:
 
 Let $\A \cellrd{V}{\theta}{W} \Cat, \quad \A \cellrd{F}{\eta}{G} \C$ be $\alpha$-natural transformations, by 
 \eqref{eq:ssss}, we have
 $$\vcenter{\xymatrix@C=2ex{V \ar@{=>}[d]_{\theta} \ar@{=>}[r]^>>>{\xi} & \C(\{V,F\}_{\beta},F-) \ar@{=>}[d]^{f^*} & \{V,F\}_{\beta} \\
 W \ar@{=>}[r]^>>>{\xi} & \C(\{W,F\}_{\beta},F-) & \{W,F\}_{\beta} \ar[u]^{\exists ! f} }}, \quad
 \vcenter{\xymatrix@C=2.5ex{W \ar@{=>}[d]_{\xi} \ar@{=>}[r]^>>>>>>>>>{\xi} & \C(\{W,G\}_{\beta},G-) \ar@{=>}[d]^{g^*} & \{W,G\}_{\beta} \\
 \C(\{W,F\}_{\beta},F-) \ar@{=>}[r]^{\eta_*} & \C(\{W,F\}_{\beta},G-) & \{W,F\}_{\beta} \ar[u]^{\exists ! g} }}$$
 A standard line of reasoning by the uniqueness in the universal properties yields that these constructions define $2$-functors
 $$(\cc{H}om_{\alpha}(\A,\Cat)^{op})_+ \xr{\{-,F\}_{\beta}} \C, \quad (\cc{H}om_{\alpha}(\A,\C))_+ \xr{\{W,-\}_{\beta}} \C,$$
 $$(\cc{H}om_{\alpha}(\A,\Cat)^{op} \times \cc{H}om_{\alpha}(\A,\C))_+ \xr{\{-,-\}_{\beta}} \C,$$
 \noindent where the subscript ``$+$'' indicates the full-subcategories with objects such that the corresponding $\beta$-limits exist.
\end{remark}

\subsection{$\eps$-ends and cotensors} \label{sub:epsends}

{\bf $\eps$-ends and $\eps$-coends}.
 
The relation of strict ends (coends) of $2$-functors 
$\A^{op} \times \A \mr{T} \B$ with weighted limits (colimits) is well understood, they are given by the weight $\A^{op} \times \A \xr{\cc{A}(-,\,-)} \cc{C}at$ 
($\A \times \A^{op} \xr{\cc{A}^{op}(-,\,-)} \cc{C}at$), see \cite[3.10]{K1}, \cite[5.2.2]{W}. However for general $\sigma$ the situation is not at all the same. Some particular cases have been considered, for example, in \mbox{\cite[9.6]{PKan1}} the pseudoend of a 
$\cc{C}at$ valued 2-functor, which requires the explicit construction of weighted pseudolimits in $\cc{C}at$, in \cite[5.3]{W} the lax coend of a $\cc{C}at$  valued 2-functor of the form 
$\A^{op} \times \A \xr{S \times T} \cc{C}at$, which requires a non-trivial change in the weight.

We will now define the $\eps$-end of a $2$-functor $\A^{op} \times \A \mr{T} \B$ (recall Notation \ref{not:eps}). The notion of $\eps$-dinatural transformation is obtained, for $\eps = \sigma$, by the requirement of invertibility on the $\theta_f$ for $f \in \Sigma$ in \S\,\ref{sub:terminology}, item \ref{laxdinat}. The case 
$\eps = s$ yields the notion of \emph{strict} dinaturality, that corresponds to $\cc{V}$-naturality when $\cc{V} = \Cat$, 
\mbox{\cite[I.3.1, I.3.5]{D},} \cite[\S$\,$2.1, \S$\,$3.10]{K1}.

\begin{\de} \label{def:cpseudocoend} 
 Let $T: \cc{A}^{op} \times \cc{A} \mr{} \cc{B}$ be a $2$-functor and $E \in \cc{B}$. 
A \emph{$\eps$-dicone} $\theta$ (with respect to $\Sigma$) for $T$ with vertex $E$ is a $\eps$-dinatural transformation from the $2$-functor which is constant at $E$ to $T$. This amounts to a lax dicone given by a family of morphisms 
\mbox{$\{E \mr{\theta_{A}} T(A,A) \}_{A \in \cc{A}}$} and a family of 2-cells \mbox{$\{T(A,f) \theta_{A} \Mr{\theta_f} T(f,B) \theta_{B} \}_{A \mr{f} B \in \cc{A}}$} such that:
\begin{enumerate}
 \item If $\eps = \sigma$, $\theta_f$ is invertible for every $f$ in $\Sigma$.
 \item If $\eps = s$, $\theta_f = id$ for every $f$.
\end{enumerate}
For each $E$, $\eps$-dicones with vertex $E$ form a category $Dicones_\eps(E,T)$, whose arrows are the morphisms as lax dicones.  

The \emph{$\eps$-end} in $\B$ (with respect to $\Sigma$) of the 2-functor $T$ is the universal $\eps$-dicone, denoted 
$$
\displaystyle \{ \eps \!\! \int_{A} T(A,A) \mr{\pi_A}  T(A,A) \}_{A\in \cc{A}}, \hspace{4ex} 
\{T(A,f) \pi_{A} \Mr{\pi_f} T(f,B) \pi_{B}\}_{A \mr{f} B \in \cc{A}}
\;.
$$ 
It is universal in the sense that for each $E \in \cc{B}$ postcomposition with $\pi$ yields an isomorphism of categories
\begin{equation} \label{eq:defcpseudoend}
\cc{B}(E, \eps \!\! \int_{A} T(A,A) ) \mr{\pi_*} Dicones_\eps(E,T) 
\end{equation}
\end{\de}
\begin{proposition} \label{def:cpseudocoendprop}
 For the $2$-functor $\A^{op} \times \A \xr{\B(E,T(-,-))} \Cat$, there is an obvious $\eps$-dicone with vertex $Dicones_{\eps}(E,T)$. It can be checked that it is universal, therefore there is an isomorphism of categories 
 \begin{equation}  \label{eq:diconoscomoend}
Dicones_\eps(E,T) \mr{\cong} \eps \!\! \int_{A}{\cc{B}(E,T(A,A))}  
 \end{equation}
As usual, then, the universal property \eqref{eq:defcpseudoend} defining $\eps \!\! \int_{A} T(A,A)$ is equivalent to stating that there is an isomorphism of categories 
 \begin{equation} \label{eq:isoparaends}
\cc{B}(E, \eps \!\! \int_{A} T(A,A))  \mr{\cong} \eps \!\! \int_{A}{\cc{B}(E,T(A,A))}  
 \end{equation}
 \noindent commuting with the $\eps$-dicones. 
 \cqd
\end{proposition} 
It is convenient to have at hand the explicit definition of the dual concept $\eps$-coend. 
\begin{definition} \label{def:cpseudocoendco}  $\eps$-coends are defined as $\eps$-ends in $\B^{op}$, for $T: \cc{A}^{op} \times \cc{A} \mr{} \cc{B}$ we define $\displaystyle \eps \!\! \int^{A} T(A,A) = \eps \!\! \int_{A} T^{op}(A,A)$, and we denote the universal $\eps$-dicone by 
$$\displaystyle \{T(A,A)\mr{\lambda_A} \eps \!\! \int^{A} T(A,A)\}_{A\in \cc{A}}, \hspace{4ex} \{\lambda_{B} T(B,f) \Mr{\lambda_f} \lambda_{A} T(f,A) \}_{A \mr{f} B \in \cc{A}}\;.
$$
\end{definition}
 A argument dual to the one  given for \eqref{eq:isoparaends} proves that the universal property defining $\eps \int^{A} T(A,A)$ can be stated as
  \begin{equation} \label{eq:isoparacoends}
 \cc{B}(\eps \!\! \int^{A} T(A,A),E)  \mr{\cong} \eps \!\! \int_{A}{\cc{B}(T(A,A),E)}
 \end{equation}

\vspace{1ex}
 
 We denote the weak concept of \emph{$\sigma$-biend} of $T$, where 
 $\pi^*$ in Definition \ref{def:cpseudocoend} is required to be only an equivalence, by $\displaystyle \sigma \! \oint_{A} T(A,A)$.

\vspace{1ex}

The following key formula remains valid for the $\cc{H}om_\eps$ categories:
\begin{\prop} \label{prop:cnatcomocend}
 For $2$-functors $P,Q: \A \mr{} \B$, we have the formula
$$
 \cc{H}om_\eps(\cc{A},\cc{B})(P,Q) = \eps \!\! \int_{A} \cc{B}(PA,QA)
$$
\end{\prop}

\begin{proof}
 It can be readily checked (by a straightforward but necessary argument) that the following data defines a universal $\eps$-dicone with vertex $\cc{H}om_\eps(\cc{A},\cc{B})(P,Q)$. 
 Projections are given by $\pi_A(\theta)=\theta_A$ for $\eps$-natural transformations $\theta$, $\pi_A(\rho)=\rho_A$ for modifications $\rho$. And the structural $2$-cells of the $\eps$-cone are given by $(\pi_f)_{\theta}=\theta_f: Qf \theta_A \Mr{} \theta_B Pf$ for $A \mr{f} B \in \A$, $\theta \in \cc{H}om_\eps(\cc{A},\cc{B})(P,Q)$. 
\end{proof}

\vspace{1ex}

\noindent {\bf Tensors and cotensors.}

\begin{definition} \label{rem:cotensors}
 $\eps$-cotensor (resp. $\eps$-tensor) are $\eps$-limits (resp. $\eps$-colimits) with $\A = 1$. In this case the choice of $\eps$ is irrelevant, since $\cc{H}om_\eps(1,\Cat)=\cc{H}om_s(1,\Cat)$ for any choice of $\eps$. We identify $1 \mr{C} \Cat$ with $C \in \Cat$, $1 \mr{B} \B$ with $B \in \B$ (note that $1^{op} = 1$), and denote cotensor products by $\{C,B\}$ and tensor products by $C \otimes B$. (see also \cite[(3.1)]{K2}). 
\end{definition}
 
Recall that in the base $2$-category $\cc{C}at$, cotensors and tensors are given by the internal hom and the cartesian product, 
$\{C,B\} = \cc{C}at(C, B)$, and 
$C \otimes B = C \times B$.

As in the case of enriched category theory, from Proposition \ref{prop:cnatcomocend} and Remark \ref{rem:limitsfuntorial} it easily follows that for any $\eps$ the 2-functor categories have cotensors and that they are computed pointwise. The proof is very similar to \cite[\S 3.3]{K1} so we omit it.
\begin{\prop} \label{prop:cotensorpointwise}
 Cotensor products are computed pointwise in $\cc{H}om_{\eps}(\A,\B)$. This means precisely that for $C \in \Cat$, $\A \mr{G} \B$, if $\{C,GA\}$ exist for each $A \in \A$ then the formula $\{C,G\}A = \{C,GA\}$ defines a $2$-functor that is the cotensor product of $C$ and $G$ in $\cc{H}om_{\eps}(\A,\B)$. \qed
\end{\prop}

We finish this subsection establishing in the $\eps$ case some well known formulas of enriched category theory. We omit to explicitly state the corresponding formulas for bilimits, which hold with the same proofs.
\begin{proposition} \label{coro:endcomoweighted}
For $\A \mr{W} \Cat$, $\A \mr{P} \B$, if $\B$ is cotensored  we have
$$
 \{ W,  P \}_{\eps} = \eps \!\! \int_{A} \{ WA , PA \}
$$ 
\end{proposition}
\begin{proof}
 We have the following chain of natural isomorphisms
 \begin{align*}
 \cc{H}om_{\eps}(\A,\Cat)(W,\B(B,P-)) & \cong \eps \!\! \int_{A} \Cat(WA, \B(B,PA)) \\
 & \cong \eps \!\! \int_{A} \cc{B}(B, \{WA, PA\}) \\
 & \cong \cc{B}(B, \eps \!\! \int_{A} \{WA, PA\}) 
 \end{align*}
 \noindent given in turn by Proposition \ref{prop:cnatcomocend}, definition of cotensor and \eqref{eq:isoparaends}. Then the statement follows  
 by Definition \ref{de:NE}. 
\end{proof}

With a dual proof we have

\begin{corollary} \label{coro:coendcomoweighted}
For $\A^{op} \mr{W} \Cat$, $\A \mr{P} \B$, if $\B$ is tensored we have
$$
 W \otimes_{\eps} P = \eps \!\! \int^{A} WA \otimes PA
$$  \qed
\end{corollary}

\vspace{-4ex}

\begin{remark} \label{rem:limitsenCat}
 For the case $\B = \Cat$, we have
\begin{enumerate}
 \item  \hspace{-1ex}  For $\A \xr{W} \Cat$, $\A \xr{P} \Cat$, 
  \hspace{-2ex} $\quad \displaystyle \{W,P\}_\eps \stackrel{\eqref{coro:endcomoweighted}}{=} \eps \!\! \int_{A} \Cat(WA,PA) \stackrel{\eqref{prop:cnatcomocend}}{=} \cc{H}om_\eps(\A,\Cat)(W,P)$
  \item For $\A^{op} \xr{W} \Cat$, $\A \xr{P} \Cat$, $\quad \displaystyle W \otimes_{\eps} P \stackrel{\eqref{coro:coendcomoweighted}}{=} \eps \!\! \int^{A} WA \times PA$
 \end{enumerate}
\end{remark}

\subsection{Conical $\sigma$-colimits} \label{sub:conicalsigma}

\begin{remark}\label{remark:conossonopnatural}
Let $F:\A \mr{} \B$ be a 2-functor, and $E$ an object of $\B$. 
It is immediate to check that the isomorphism 
in \S$\,$\ref{sub:terminology}, item \ref{item:conicalweighted}, restricts to an isomorphism 
(recall that $k_X$ stands for the functor constant at $X$)
\begin{equation} \label{eq:item9parac}
 \cc{H}om_{\sigma}(\A^{op},\Cat)(k_1,\B(F-,E)) \cong \cc{H}om_{op \sigma }(\A,\B)(F,k_E)
\end{equation}
By considering the weight $k_1$ in \eqref{eq:defcpseudocolimconnu}, it follows
$$
\xymatrix{\B(k_1 \otimes_{\sigma} F, E) \ar[r]^<<<<<{\nu^*}_<<<<<{\cong} & \cc{H}om_{\sigma}(\A^{op},\Cat)(k_1,\B(F-,E)) \; {\cong} \; \cc{H}om_{op \sigma }(\A,\B)(F,k_E)} 
$$
\end{remark}

\begin{\de} \label{def:claxlimit} 
Let $F:\A \mr{} \B$ be a 2-functor, and $E$ an object of $\B$. We define the category of \emph{$\sigma$-cones}, \mbox{$Cones_\sigma^{\Sigma}(F,E) = Cones_\sigma^{k_1}(F,E)$}, see Definition \ref{de:NE}. 
By Remark \ref {remark:conossonopnatural}, a \emph{$\sigma$-cone} for $F$ (with respect to $\Sigma$) with vertex $E$ corresponds to a op$\sigma$-natural transformation $F \Mr{\theta} k_E$, this amounts to a
lax cone $\{FA  \mr{\theta_A} E\}_{A\in \cc{A}}$,
\mbox{$\{\theta_{B} Ff \Mr{\theta_f} \theta_A \}_{A\mr{f} B \in \A}$} such that $\theta_f$ is invertible for every $f$ in $\Sigma$. 
The morphisms between two $\sigma$-cones correspond to their morphisms as lax cones. 

We now describe the universal property  defining the (conical) \mbox{$\sigma$-colimit} of $F$. 
The \mbox{\emph{$\sigma$-colimit}} in $\B$ (with respect to $\Sigma$) of the $2$-functor \mbox{$F:\A\mr{}\B$} is the universal $\sigma$-cone, denoted \mbox{$\{FA \mr{\lambda_A} \coLimconw{A\in \A}{FA}\}_{A\in \A}$,} \mbox{$\{\lambda_{B} Ff \Mr{\lambda_f} \lambda_A \}_{A\mr{f} B \in \A}$} in the sense that for each \mbox{$E\in \B$,} pre-composition with $\lambda$ is an isomorphism of categories
\begin{equation}\label{isoplim}
\; \B(\coLimconw{A\in \A}{FA},E) \mr{\lambda^*} Cones_\sigma(F,E) 
\end{equation}
We denote the weak notion of \mbox{\emph{(conical) $\sigma$-bicolimit}} as in Remark \ref {rem:defbi}, where $\lambda^*$ is an equivalence, 
$\bicoLimconw{A\in \A}{FA}$. By definition we have, for $F:\A\mr{}\B$,

\begin{equation} \label{eq:bydefwehave}
\coLim{A\in \A}{FA} = k_1 \otimes_{\sigma} F, \hspace{2cm} \bicoLim{A\in \A}{FA} = k_1 \;{}_{bi} \! \otimes_{\sigma} F.
\end{equation}
\end{\de}

 Conical op$\sigma$-colimits are special cases of \emph{Cartesian quasi limits} considered by J. W. Gray in \cite[I,7.9.1 iii)]{GRAY}. What we would denote by $\opcoLimconw{A\in \A}{FA}$ is, in Gray's notation, $Cart\ q \hbox{-}\coLimsinc{\A,iso\Sigma}{F}$.
\begin{remark}\label{rem:notgray} 
 Conical $\sigma$-limits $\Limconw{A\in \A}{FA}$ are $\sigma$-limits weighted by $k_1$, in this case there is no ``op'' in the lax naturality involved (because there is no ``op'' in the first isomorphism of \S$\,$\ref{sub:terminology}, item \ref{item:conicalweighted}), and they correspond in Gray's notation to $Cart\ q\hbox{-}\Limsinc{\A,iso\Sigma}{F}$.
 We also refer to $\sigma$-natural transformations $k_E \Mr{\theta} F$ as \emph{$\sigma$-cones} (as in Definition \ref {def:claxlimit} above), 
and denote the so-obtained category of $\sigma$-cones by  
 $Cones_\sigma^{\Sigma}(E,F)$.
 \end{remark}
 The following diagram illustrates the correspondence between $\sigma$-cones in $\B^{op}$ and $\sigma$-cones in $\B$ (recall that $2$-cells in $\B^{op}$ keep their direction and that we denote objects in $\B^{op}$ with an overline, see \S$\,$\ref{sub:terminology}, item \ref{def:op}):
\begin{equation} \label{eq:conescocones}
\hbox{For } A \mr{f} B \hbox{ in } \A, \quad
\vcenter{\xymatrix@R=.5pc{ & \overline{FB} \ar[dd]^{Ff} \\ \overline{E} \ar[rd]_{\theta_A} \ar[ru]^{\theta_B} \ar@{}[r]|>>>>{\Downarrow \theta_{f}} & \\ & \overline{FA} }}
\hbox{\; in } \B^{op} \hbox{ \quad corresponds to \quad }
\vcenter{\xymatrix@R=.5pc{ FA \ar[dd]_{Ff} \ar[rd]^{\theta_A} \\ & E \ar@{}[l]|>>>>{\Uparrow \theta_{f}} \\ FB \ar[ru]_{\theta_B} }}
\hbox{\; in } \B.
\end{equation}
\begin{\de} \label{def:cartsigma}
Recall Definition \ref {def:gammaP}. We denote by $\cart_\Sigma$ the $1$-subcategory of $\cc{E}l_P$ with arrows $(f,\varphi)$ that satisfy $f \in \Sigma$, $\varphi$ invertible (i.e. the intersection of $\cart_P$ and $\Diamond_P^{-1}(\Sigma)$).
\end{\de}

In \cite[Theorem 15]{S} Street shows that for each weight $\A \mr{W} \Cat$, $s$-limits weighted by $W$ are equivalent to a special type of Gray's cartesian quasi-limit (see Remark \ref {rem:notgray}) over $\cc{E}l_{W}$.
A slight modification of this procedure shows that weighted $\sigma$-limits can be expressed as conical $\sigma$-limits. Since we will use this result for colimits, we prefer to prove the colimit version.

\begin{\prop} \label{key}
  Let
 $\A^{op} \mr{W} \Cat$, $\A \mr{P} \B$, then we have
$$
  \cc{H}om_{\sigma}^{\Sigma}(\A^{op},\Cat)(W,\B(P-,B)) \cong \cc{H}om_{\sigma}^{\cart_\Sigma}(\cc{E}l_{W},\Cat)(k_1, \B(P\Diamond_W^{op} -,B))
$$
\end{\prop}

\begin{proof}
 Consider the isomorphism
 $$\cc{H}om_{\ell}(\A^{op},\Cat)(W,H) \cong \cc{H}om_{\ell}(\cc{E}l_{W},\Cat)(k_1, H \Diamond_W)$$
 \noindent of Remark \ref {rem:oplaxdense} and the explicit formulas therein. 
 Then it can be seen at once that the isomorphism restricts to 
 $$\cc{H}om_{\sigma}^{\Sigma}(\A^{op},\Cat)(W,H) \cong \cc{H}om_{\sigma}^{\cart_\Sigma}(\cc{E}l_W,\Cat)(k_1, H \Diamond_W)$$
 In particular, for $H = \B(P-,B)$ we have the desired isomorphism.
\end{proof}

This Proposition has as a corollary the following fundamental result:

\begin{theorem} \label{colimdaconico}
   Let $\A^{op} \mr{W} \Cat$, $\A \mr{P} \B$, then
$$
 W \otimes_{\sigma} P =  k_1 \otimes_{\sigma} P\Diamond_W^{op} = \coLimconcartsigma{(x,A) \in \cc{E}l_W^{op}} {PA} ,  \quad 
 W \;{}_{bi} \! \otimes_{\sigma} P = k_1 \;{}_{bi} \! \otimes_{\sigma} P\Diamond_W^{op} = \bicoLimconcartsigma{(x,A) \in \cc{E}l_W^{op}} {PA},
$$
which means that the universal properties defining each object are equivalent.
In particular, by considering $\Sigma = \A_0$, we have
$$
 W \otimes_{p} P = \coLimconcartw{(x,A) \in \cc{E}l_W^{op}} {PA},  \quad\quad  W \;{}_{bi} \! \otimes_{p} P = \bicoLimconcartw{(x,A) \in \cc{E}l_W^{op}} {PA}.
$$
\end{theorem}
\begin{proof}
 The first equality (in both expressions) is given by Proposition  \ref{key}, the second one is \eqref{eq:bydefwehave}.
\end{proof}

 In particular for $\eps = p$ using Theorem \ref{colimdaconico} we have the following expressions of a pseudocoend of tensors as a conical 
 $\sigma$-colimit (for the first equality consider 
$(\cc{A}^{op})^{op} \mr{P} \cc{C}at$, 
$\cc{A}^{op} \mr{W} \cc{C}at$, and the first coend as indexed by 
$\cc{A}^{op}$).

\begin{proposition} \label{eq:paraintro}
$$ 
  \coLimconcartp{(x,A) \in \cc{E}l_P^{op}} {WA} = \displaystyle p \!\! \int^{A} PA \times WA  = \displaystyle p \!\! \int^{A} WA \times PA = \coLimconcartw{(y,A) \in \cc{E}l_W^{op}} {PA}
$$

\vspace{-4ex}

\qed
\end{proposition}

\vspace{1ex}

{\bf The expression of a $\Cat$-valued $2$-functor as a conical $\sigma$-bicolimit of representable $2$-functors.} 

It is a classical result that any $\cc{S}et$-valued functor has a canonical expression as a colimit of representable functors. 
We now establish a $2$-categorical version of this result.
 Consider a $2$-functor $\A \mr{P} \Cat$, and the Yoneda embedding \mbox{$\A^{op} \mr{h} \cc{H}om_p(\A,\Cat)$,} \mbox{$hA = \A(A,-)$.} 
  Recall the Pseudo-Yoneda Lemma, \cite[(1.9)]{S2}, see \mbox{\cite[1.1.18]{DD}} for a proof: 
  
\begin{sinnadaitalica}[\bf Pseudo-Yoneda Lemma] $ $ \label{pseudoyoneda}

{\bf a)} For any $2$-functor $\A \mr{Q} \Cat$, evaluation at the identity for each $A \in \cc{A}$ provides the components: 
 $$\cc{H}om_p(\A,\Cat)(\A(A,-),Q) \mr{\approx} QA$$
of a pseudo-equivalence, that is an equivalence in 
$\cc{H}om_p(\A,\Cat)$, between $Q$ and the $2$-functor on the left side. Furthermore, this equivalence is pseudonatural in the variable $Q$.

From this, as usual, it follows: 
  
{\bf b)} The Yoneda embedding is pseudo-fully-faithful. That is, there is an equivalence of categories:
$$
\cc{H}om_p(\A,\Cat)(\A(A,-),\, \A(B,-)) \mr{\approx} \cc{A}(B,\, A).
$$

\vspace{-4ex}

\cqd
\end{sinnadaitalica}
\begin{proposition} \label{Pcomobicoend}
For any $2$-functor $\A \mr{P} \Cat$, we have a pseudo-equivalence, that is an equivalence in $\cc{H}om_p(\A,\Cat)$:
$$
\displaystyle P \;\;\approx \;\; P \otimes_p h  \;\stackrel{(1)}{=}\; p \!\! \int^{A} PA \otimes \cc{A}(A,-).
$$
\end{proposition}
\begin{proof}
Consider $P$ as a weight for the colimit of the Yoneda embedding 
$h$, then Corollary \ref{coro:coendcomoweighted} shows $\stackrel{(1)}{=}$. 
Then, we have the following chain of equivalences, pseudonatural in the variable $Q$:
\begin{align*}
 \cc{H}om_p(\cc{A},\cc{C}at) (p \!\!\int^{A} PA \otimes \cc{A}(A,-),Q)  
 & \cong \; p \!\!\int_{A} \cc{H}om_p(\cc{A},\cc{C}at)(PA \otimes \cc{A}(A,-) ,Q) 
 \\
 & \cong \; p \!\!\int_{A} \Cat(PA, \cc{H}om_p(\cc{A},\cc{C}at)(\cc{A}(A,-),Q)) 
 \\
 & \approx \; p \!\!\int_{A} \cc{C}at(PA, \, QA)
 \\
 & \cong \; \cc{H}om_p(\cc{A},\, \cc{C}at)(P,\,Q). 
\end{align*}
\noindent justified, in turn,  by 
\eqref{eq:isoparacoends}, 
Proposition \ref {prop:cotensorpointwise}, 
Pseudo Yoneda a) and  
Proposition \ref{prop:cnatcomocend}. 
By the pseudonaturality in $Q$, a use of Pseudo Yoneda b), applied this time to the category $\cc{H}om_p(\A,\Cat)$, finishes the proof.
\end{proof}

From this Proposition and Theorem \ref{colimdaconico} we have:
\begin{proposition} \label{Pcomosigmabicolimit}
For any $2$-functor $\A \mr{P} \Cat$, we have a pseudo-equivalence, that is an equivalence in $\cc{H}om_p(\A,\Cat)$:
$$ 
\displaystyle P \approx  \coLimconcartp{(x,A) \in \cc{E}l_P^{op}} \A(A,-).
$$

\vspace{-5ex}

\cqd
\end{proposition}

\begin{remark} \label{rem:diagenHoms}
In particular, since $\sigma$-bilimits are defined up to equivalence, it follows that any $\cc{C}at$-valued $2$-functor $P$ is a conical $\sigma$-bilimit in $\cc{H}om_p(\cc{A},\cc{C}at)$ of a $2$-diagram in $\cc{H}om_s(\cc{A},\cc{C}at)$ of representable $2$-functors,  indexed by the pair 
($\cc{E}l_P^{op}$,$\,\cart_P$).
\end{remark} 

\subsection{A construction of conical $\sigma$-colimits of categories.} 
\label{constructions}

In \cite[I,7.11.4 i)]{GRAY} Gray proves that conical op$\sigma$-colimits in $\Cat$ exist and gives an explicit construction of them. In Proposition \ref{prop:calculodeclaxcolim} below, we interpret this result according to our notation. We will use the left adjoint $\pi_0$ of the inclusion $\Cat \mr{d} 2$-$\Cat$ 
(where $2$-$\Cat$ is the $2$-category of small $2$-categories, $2$-functors and $2$-natural transformations)
and the existence of the usual category of fractions \cite{GZ}. For a subcategory $\Sigma$ of a category $C$, we will denote this category by $C[\Sigma^{-1}]$. 

We observe that Gray only makes invertible those morphisms of the form $(f,id) \in \cart_\Sigma$ while we invert every morphism in $\cart_\Sigma$. Since every morphism $(f,\varphi) \in \cart_\Sigma$ can be factorized as $(id,\varphi) (f,id)$ and $(id,\varphi)$ is already invertible because $\varphi$ is an isomorphism, both constructions are isomorphic.    

\begin{proposition} \label{prop:calculodeclaxcolim}
Let $\A \mr{Q} \Cat$ be a $2$-functor. Then 
$$
\opcoLim{A \in \A}{QA} = (\pi_0\cc{E}l_Q)[\cart_\Sigma^{-1}]
$$.

\vspace{-7ex}

\cqd
\end{proposition}

\begin{remark} \label{rem:calculodecoplaxcolim}
 Let $\A \mr{Q} \Cat$ be a $2$-functor. 
 There is a construction \emph{dual} to the \mbox{$2$-category} of elements $\cc{E}l_Q$, that we will denote by $\Gamma_{Q}$, where the direction of $\varphi$ in Definition \ref {def:gammaP} is reversed. 
 Note that Definition \ref {def:cartsigma} can be easily adapted to define the $1$-subcategory $\cart_\Sigma$ of $\Gamma_{Q}$.
 Then a proof \emph{dual} to the one of Proposition \ref {prop:calculodeclaxcolim} 
 shows that, for a $2$-functor $\A \mr{Q} \Cat$, 
$$
 \coLim{A \in \A}{QA} = (\pi_0\Gamma_{Q}^{op})[\cart_\Sigma^{-1}]              
$$
 
 \vspace{-1ex}
 
 The interested reader can also see \cite[Proposition 1.17]{Bird} for the $\ell$-case, and \mbox{\cite[Theorem 5.2]{Data}} for the $p$-case, where the following formula is established:  
$$
\coLim{A \in \A}{QA} = \pi_0(\Gamma_{Q}^{op}[\cart_\Sigma^{-1}]),
$$

\vspace{-1ex}

\noindent In the previous formula, $\Gamma_{Q}^{op} [\cart_\Sigma^{-1}]$                                                                                                                                              is the $2$-category of fractions in the sense of \cite{P}.
Note that from the adjunction $\pi_0 \dashv d$ for any $1$-subcategory $\Sigma$ of a $2$-category $\A$ it follows $(\pi_0\A)[\Sigma^{-1}] = \pi_0(\A[\Sigma^{-1}])$, thus the two constructions are the same. 
\end{remark}

\begin{remark} \label{rem:lambdadelclaxcolim}
 Note that, since computing $\pi_0$ and the category of fractions doesn't change the objects, the objects of $\coLim{A \in \A}{QA}$ can be taken to be the objects of $\Gamma_Q$, which are 
 pairs $(x,A)$ with $A \in \A$, $x \in QA$. 
 By looking at the proof of \cite[I,7.11.1]{GRAY} (which has \mbox{\cite[I,7.11.4 i)]{GRAY}}, i.e. Proposition \ref {prop:calculodeclaxcolim} as corollary), we have a formula for the universal \mbox{$\sigma$-cone} $\lambda$, in particular on objects 
 $\lambda_A(x) = (x,A)$. Note that for each object 
 $c \in \coLim{A \in \A}{QA}$, there are $A \in \A$, $x \in QA$ such that $\lambda_A(x) = c$. 
\end{remark}

\begin{lemma} \label{lema:elemenplimconpeso}
   Let  $\A^{op} \mr{W} \Cat$, $\A \mr{P} \Cat$, and 
   consider the universal $w$-$\sigma$-cone \mbox{$W \Mr{\nu} \Cat(P-,C)$,} where $C = W \otimes_{\sigma} P$ (see \eqref{eq:defcpseudocolimconnu}), note that $\nu$ is a $\sigma$-natural transformation. 
   Then for each object $c \in C$, there exist $A \in \A$, $x \in WA$, $a \in PA$ such that $\nu_A(x)(a) = c$.
\end{lemma}
\begin{proof}
By Theorem \ref{colimdaconico} we have 
$\,C  = \coLimconcartsigma{(x,A) \in \cc{E}l_W^{op}} {PA}$. In other words, $W \otimes_{\sigma} P$ is the $\sigma$-colimit of the \mbox{$2$-functor} $\cc{E}l_W^{op} \mr{\Diamond_W^{op}} \A \mr{P} \Cat$. We may compute this colimit using Remark \ref {rem:calculodecoplaxcolim}, then we have the colimit $\sigma$-cone $P \Diamond_W^{op} \Mr{\lambda} k_C$. By Remark \ref {rem:lambdadelclaxcolim}, there are $(x,A) \in \cc{E}l_W$, $a \in PA$ such that $\lambda_{(x,A)}(a) = c$.
 
 Now, as in the proof of Theorem \ref{colimdaconico}, the correspondence between the \s-colimit $\sigma$-cone $P \Diamond_W^{op} \Mr{\lambda} k_C$ and 
 the \s-colimit $w$-$\sigma$-cone  $W \Mr{\nu} \Cat(P-,C)$ is given by Proposition \ref{key} and \eqref{eq:item9parac}, and therefore by the formulas in Remark \ref{rem:oplaxdense} and in \S$\,$\ref{sub:terminology}, item \ref{item:conicalweighted}. Then we have $\nu_A(x)(a) = \lambda_{(x,A)}(a) = c$.
\end{proof}

\subsection{Pointwise limits} \label{sub:pointwise}

We analyze now the computation of weighted $\varepsilon$-limits in functor categories. The 
pointwise computation of arbitrary weighted $\varepsilon$-limits is a much more delicate matter than that of cotensors (Proposition \ref{prop:cotensorpointwise}), we give below a general result regarding the pointwise computation of $\alpha$-limits in $op \beta$-functor categories (with pseudo or strict diagrams). Note in particular the appearance of the ``op" prefix, this is reminiscent of the lifting of op-lax limits to the \mbox{$2$-category} of strict algebras and lax morphisms for a $2$-monad (\cite{Lack}), which has as a particular case the case $\gamma = \ell$ of Proposition \ref{prop:claxcolimptoapto}. 

\begin{remark} \label{rem:evfuntor}
 Let $\B,\C$ be $2$-categories, and 
 $\gamma \in \{s,\, p,\, \ell\}$. Then we have a $2$-functor 
 \mbox{$\B \xr{ev_{(-)}} \cc{H}om_\gamma(\cc{H}om_{op \gamma }(\B,\C),\C)$} given by the formulas, for $B \cellrd{f}{\mu}{g} B'$ in $\B$, $F \cellrd{\theta}{\rho}{\eta} G$ in $\cc{H}om_{op \gamma }(\B,\C)$,
 
 \begin{enumerate}
  \item  \label{one} $ev_B (F) = FB, \quad ev_B(\theta) = \theta_B, \quad ev_B(\rho) = \rho_B$
  \item $(ev_f)_F = Ff, \quad (ev_f)_{\theta} = \theta_f$
  \item $(ev_\mu)_F = F\mu$
 \end{enumerate}
 For each $B \in \cc{B}$ and any $\eps$ the definition in \ref{one}. determines $2$-functors 
$$
\cc{H}om_{\eps}(\B,\C) \xr{ev_B} \cc{C}, \hspace{3ex}
   \cc{H}om_{op\eps}(\B,\C) \xr{ev_B} \cc{C},
$$
that is, functors
$$
\cc{H}om_{\eps}(\B,\C)(F,\,G) \xr{ev_B} \cc{C}(FB,\, GB), \hspace{3ex}
\cc{H}om_{op\eps}(\B,\C)(F,\,G) \xr{ev_B} \cc{C}(FB,\, GB).
$$
All the verifications are straightforward. \cqd
\end{remark}

 \vspace{1ex}
 
Recall Notation \ref{epsilonnotation}. Given $2$ categories $\A,\; \B$, in the next proposition we let \mbox{$\alpha \in \Lbl_\A$,} 
$\beta \in \Lbl_\B$ be the label ``$s$", or labels corresponding to arbitrary 
$1$-subcategories of $\A$, $\B$ respectively. 
Among all the possible labels, the three labels 
$\ell,p,s$  always make sense for any $\A$ and $\B$. We will use the letter $\gamma$ to refer to these labels. With this in mind, we have:
\begin{proposition}\label{prop:claxcolimptoapto}
Let $\gamma \in \{\ell, p, s\}$, $\alpha \in  \Lbl_\A$, $\beta \in \Lbl_\B$, such that $\alpha \geq \gamma$, $\beta \geq \gamma$.
Then weighted $\alpha$-limits of $op\gamma$-diagrams are computed pointwise in the \mbox{$2$-functor} \mbox{$2$-categories} $\cc{H}om_{op\beta}(\B,\C)$ (in particular in $\cc{H}om_{op\gamma}(\B,\C)$), and are preserved by the inclusion $2$-functor $\cc{H}om_{op\gamma}(\B,\C) \mr{i} \cc{H}om_{op\beta}(\B,\C)$. This means precisely that given \mbox{$2$-functors} $\A \mr{W} \Cat$, \mbox{$\A \mr{F} \cc{H}om_{op\gamma}(\B,\C)$}, if the $\alpha$-limits $\{W,ev_B F\}_{\alpha}$ exist in $\C$ for each $B \in \B$, the definition $LB = \{W,ev_B F\}_{\alpha}$ determines a \mbox{$2$-functor} $\B \mr{L} \C$ which is the $\alpha$-limit $L = \{W,iF\}_{\alpha}$  of $iF$ weighted by $W$ in $\cc{H}om_{op\beta}(\B,\C)$. 
Denoting the composition $ev_B F$ by $(F-)B: \A \mr{} \C$,
we can write \mbox{$\{W,\, F\}_\alpha (B) = \{W, \, (F-)(B)\}_\alpha$.} Note that when $\gamma = \ell$, this forces that also $\alpha = \ell$ and $\beta = \ell$.  
\end{proposition}

\begin{proof}
 The definition of $L$ is given by the composition (see Remarks \ref{rem:limitsfuntorial} and \ref{rem:evfuntor}):
 $$
 L: \B \xr{ev_{(-)}} \cc{H}om_{\gamma}(\cc{H}om_{op\gamma}(\B,\C),\C) \xr{F^*} \cc{H}om_{\gamma}(\A,\C)_+ \xr{\{W,-\}_{\alpha}} \C
 $$ 
Note that by hypothesis the limits $\{W,ev_B F\}_{\alpha}$ exist for each $B$. It follows then that the composite $F^* ev_(-)$ actually lands in $\cc{H}om_{\gamma}(\A,\C)_+$, so $L$ is defined. 

 Clearly $LB = \{W,ev_B F\}_\alpha$.  For each $B \in \cc{B}$, let $W \Mr{\xi_B} \cc{C}(LB,\, (F-)B)$ be a $\alpha$-limit  $w$-$\alpha$-cone in $\cc{C}$, $LB = \{W,\, (F-)B\}_\alpha$, let us denote the components of $\xi_B$ by \mbox{$WA \xr{\xi_{B, A}} \cc{C}(LB,\, (FA)B)$.} Then:

\vspace{1ex}
 
 a) \emph{For each $A \in \cc{A}$, $WA \xr{\xi_{B, A}} \cc{C}(LB,\, (FA)B)$ are the components of an $s$-dinatural cone in the variable $B$.}
 
 \emph{proof:} Let $B \mr{h} B'$ in $\cc{B}$, and consider the following diagram as in the proof of Remark \ref{rem:limitsfuntorial}:
$$
\xymatrix@C=10ex
     {
        W \ar@{=>}[d]_{\xi_{B}} \ar@{=>}[r]^>>>>>>>>>{\xi_{B'}} 
      & \C(LB',\, (F-)B') \ar@{=>}[d]^{(Lh)^*} 
      & LB' 
      \\
        \C(LB,\, (F-)B) \ar@{=>}[r]^{((F-)h)_*} 
      & \C(LB,\, (F-)B') 
      & LB \ar[u]^{\exists ! \, Lh}
     }
$$ 
Then, evaluating at $A$ finishes the proof. 
\hfill \emph{end proof of a)} $\Box$
 
  \vspace{1ex}
  
  b) \emph{The arrows in a) determine a $w$-$\alpha$-cone:}
   
 $WA \mr{\xi_A} \cc{H}om_{s}(\B,\C)(L,\,FA) \mr{i} 
                                       \cc{H}om_{op\beta}(\B,\C)(L,\,FA)$, \;
 $W \Mr{\xi} \cc{H}om_{op\beta}(\B,\C)(L,\,F)$.
 
 \emph{proof:} Let $A \mr{g} A'$ in $\cc{A}$, and consider the following diagram:
 $$
 \xymatrix@C=5ex
    {
      WA \ar[r]^>>>>>>{\xi_A} 
         \ar[d]_{Wg} 
         \ar@{}[dr]|{\Downarrow \xi_g} 
    & \cc{H}om_{op\gamma}(\B,\C)(L,\,FA) 
         \ar[d]^{(Fg)_*} 
         \ar[r]^<<<<<{i}
         \ar@{}[dr]|{\equiv}
    & \cc{H}om_{op\beta}(\B,\C)(L,\,FA) 
         \ar[d]^{(Fg)_*} 
         \ar[r]^<<<<<{ev_B}
         \ar@{}[dr]|{\equiv}
    & \cc{C}(LB,\, (FA)B) 
         \ar[d]^{((Fg)_B)_*} 
    \\ 
      WA' \ar[r]^>>>>>>{\xi_{A'}} 
    & \cc{H}om_{op\gamma}(\B,\C)(L,\,FA') 
         \ar[r]^<<<<<{i}
    & \cc{H}om_{op\beta}(\B,\C)(L,\,FA') 
         \ar[r]^<<<<<{ev_B}
    & \cc{C}(LB,\, (FA')B)
    } 
 $$
 By Proposition \ref{prop:cnatcomocend} it follows there are arrows $\xi_A$, 
 $\xi_A'$ such that $ev_B \, \xi_A = \xi_{B,A}$, 
 \mbox{$ev_B \, \xi_{A'} = \xi_{B,A'}$}. By hypothesis, for each 
 $B$, $\xi_{B}$ is $\alpha$-natural (in the variable $A$). Thus there is a $2$-cell
  $((Fg)_B)_* \, \xi_{B,A} \Mr{(\xi_B)_{g}} \xi_{B,A'} \, Wg$.
  It is straightforward to check that $((Fg)_B)_* \, \xi_{B,A}$ and   $\xi_{B,A'} \, Wg$ are $op \gamma$-dicones for the $2$-functor 
$\cc{C}(L-,\, (FA')-)$, and the $(\xi_B)_g$ determine a morphism of dicones. Then, the existence of $\xi_g$ as indicated in the diagram,  
  \mbox{$(\xi_B)_g = ev_B \, \xi_g$,} follows from the isomorphism of categories 
$$
\cc{C}at(WA,\, \cc{H}om_{op\gamma}(\B,\C)(L-,\,(FA')-)) \mr{\cong} 
                                \cc{D}icones_{op\gamma}(WA, \, \cc{C}(L, FA')).
$$ 
This shows we have a $w$-$\alpha$-cone:
$W \Mr{\xi} \cc{H}om_{op\gamma}(\B,\C)(L,\,F-)$, thus also one into 
$\cc{H}om_{op\beta}(\B,\C)(L,\,F-)$.

The axioms of \mbox{$\alpha$-naturality} for $\xi_g$ can be checked using the corresponding axioms for $(\xi_B)_g$ and the isomorphism of categories above.    
\hfill \emph{end proof of b)} $\Box$

 \vspace{1ex}                                                           
 
 c)  \emph{The $w$-$\alpha$-cone in b) is a $\alpha$-limit cone in 
 $\cc{H}om_{op\beta}(\B,\C)$, $L = \{W,\, F\}_\alpha$.} 
 
 \emph{proof:} It only remains to show the universal property. Let 
 $\cc{B} \mr{H} \cc{C}$ be a $2$-functor and  \mbox{$W \Mr{\rho} \cc{H}om_{op\beta}(\B,\C)(H,\,F-)$} be a $w$-$\alpha$-cone. We have:
$$
\xymatrix@C=10ex
     {
        W \ar@{=>}[rd]+<-9ex, +1ex>^>>>>>>>>{\rho_{B}} 
          \ar@{=>}[r]^>>>>>>>>>{\xi_{B}} 
      & \C(LB,\, (F-)B) \ar@{=>}[d]^{(\eta_B)^*} 
      & LB 
      \\
      &  \C(HB,\, (F-)B)   
      & HB \ar[u]^{\exists ! \, \eta_B}
     }
$$  
We now prove that $\eta_B$ is $op\beta$-natural in the variable $B$. Let 
$B \mr{h} B'$ in $\cc{B}$. Consider the isomorphism in the definition of $\xi_{B'}$, 
$$
(1) \hspace{10ex} \cc{C}(HB,\, LB') \xr{(\xi_{B'})^*} \cc{H}om_\alpha(\cc{A},\, \cc{C}at)(W,\, \cc{C}(HB,\, (F-)B'))      
$$
We have the $\alpha$-natural structural $2$-cell $\eta_h$ defined as follows:

\vspace{1ex}

$(Hh)^*(\eta_{B'})^*\xi_{B'} = (Hh)^*\rho_{B'} \;\Mr{\rho_h}\; ((F-)h)_*\rho_B = 
((F-)h)_*(\eta_B)^*\xi_B = $ 

\hfill $= (\eta_B)^*((F-)h)_*\xi_B = (\eta_B)^*(Lh)^*\xi_{B'}$.

\vspace{1ex}

We suggest the reader to use the diagram below to check the equations in this definition.
$$
\xymatrix@C=10ex@R=8ex
  {
   & \cc{C}(HB,\, (F-)B) \ar@/^4ex/@<1.5ex>[rrdd]^{((F-)h)_*}
   \\
   & \cc{C}(LB,\, (F-)B) \ar@/^2ex/@<1.5ex>[rd]^{((F-)h)_*}
                         \ar[u]_{(\eta_B)^*}
   \\
   W \ar@/^4ex/@<1.5ex>[ruu]+<-7.5ex, -0.5ex>^{\rho_B}
     \ar@/^1ex/@<1.5ex>[ru]+<-7.5ex, -0.8ex>^{\xi_B}
     \ar@/_1ex/@<-1.5ex>[rd]+<-8ex, +1ex>^{\xi_{B'}}
     \ar@/_4ex/@<-1.5ex>[rdd]+<-8ex, +1ex>^{\rho_{B'}}
     \ar@/_6ex/@<-1.5ex>[rddd]+<-8ex, +1ex>^{\rho_B}
   && \cc{C}(LB,\, (F-)B')  \ar[r]^{(\eta_{B})^*}
   & \cc{C}(HB,\, (F-)B')
   \\
   & \cc{C}(LB',\, (F-)B') \ar[d]^{(\eta_{B'})^*}
                           \ar@/_2ex/@<-1.5ex>[ru]^{(Lh)^*}
   & {\Uparrow (\eta_h)^*}
   \\
   & \cc{C}(HB',\, (F-)B') \ar@/_5ex/@<-1.5ex>[rruu]+<-4ex,               
                                                -0.5ex>^{(Hh)^*}
                           \ar@{}[d]|{\Downarrow \rho_h}
   \\
   & \cc{C}(HB,\, (F-)B)  \ar@/_7ex/@<-1.5ex>[rruuu]+<0ex,  
                                             -0.5ex>^{((F-)h)_*}
  }
$$
Thus, we have a $2$-cell $(\eta_{B'}\,Hh)^*\,\xi_{B'} \;\Mr{\rho_h}\; (Lh\,\eta_B)^*\,\xi_{B'}$.
By the isomorphism (1) above, it follows that there exist a unique
$
\xymatrix
     {
       HB \ar[r]^{\eta_B} \ar[d]_{Hh} 
     & LB \ar[d]^{Lh}
     \\
       HB' \ar[r]^{\eta_{B'}}
           \ar@{}[ur]|{\Uparrow \eta_h}
     & LB' 
     }
$
such that $\rho_h = (\eta_h)^* \xi_{B'}$.
The $\beta$-naturality axioms for $\eta$ follow from the $\beta$-dicone axioms for $\rho$ and the $op\beta$-naturality of $\rho_A(x)$, $A \in \cc{A}$, $x \in WA$. We leave to the reader the verification of the $2$-dimensional aspect of the universal property. \hfill \emph{end proof of c)} $\Box$

\vspace{1ex}

 This finishes the proof of the proposition. Note that if $\C$ has tensor products with \mbox{{\bf 2}$\; = \{0 \to 1\}$,} by Proposition \ref{prop:cotensorpointwise} so does $Hom_{op\beta}(\B,\C)$ and thus (as in \cite[p.306]{K2}) the $2$-dimensional aspect of the universal property follows from the $1$-dimensional one. However we think it is pertinent not to assume that $\C$ has tensors, for example, in practice, $\C$ may only have all finite conical $p$-limits. 
\end{proof}
\begin{remark} \label{limitinco}
Note that to compute pointwise $\alpha$-limits in the 
$\cc{H}om_{op\beta}(\cc{B},\, \cc{C})$ categories we use $\alpha$-limits in 
$\cc{C}$. Since $\cc{H}om_{op \beta }(\B,\C) \cong \cc{H}om_{\beta }(\B^{co},\C^{co})$ \mbox{(\S$\,$\ref{sub:terminology}, item \ref{item:coco}),} to compute $\alpha$-limits in $\cc{H}om_{\beta}(\B,\C)$ we use $\alpha$-limits in $\C^{co}$, that is 
$op\alpha$-limits in $\C$ (Remark \ref {rem:4en1}).
\end{remark}
Since for $\beta = $ $p$ or $s$  we have isomorphisms $\cc{H}om_{op \beta }(\B,\C) \cong \cc{H}om_{\beta }(\B,\C)$, it follows: 
\begin{corollary}\label{coro:claxcolimptoapto}
Weighted $\sigma$-limits are computed pointwise in the $2$-functor $2$-categories $\cc{H}om_s(\B,\C)$ and $\cc{H}om_p(\B,\C)$. 
The inclusion $\cc{H}om_s(\B,\C) \mr{i} \cc{H}om_p(\B,\C)$ preserves these limits, we have $i \{W,F\}_\sigma = \{W,iF\}_\sigma$. \qed
\end{corollary}

\begin{remark} \label{rem:unosiunono}
 In general, $s$-limits are not computed pointwise in $\cc{H}om_p(\A,\B)$, see \mbox{\cite[Example 6.2]{BKP}} for a counterexample. The obstruction in the proof of Proposition \ref{prop:claxcolimptoapto} if one tries to prove this statement is that the definition $LB$ in the beginning of the proof would not be functorial in the variable $B$, as we do not have a $2$-functor $\cc{H}om_p(\A,\C) \xr{\{W,-\}_s} \C$. 
 \end{remark} 

\subsection{Interchange formulas}
\label{sub:interchange}

As usual, the commutativity of limits with limits follows from the pointwise computation. Recall the notation considered before Proposition \ref{prop:claxcolimptoapto}.

\begin{proposition} \label{sin:ginterchange}  
Let $\gamma \in \{\ell, p, s\}$, $\alpha \in  \Lbl_\A$, $\beta \in \Lbl_\B$, such that $\alpha \geq \gamma$, $\beta \geq \gamma$. 
Let $\cc{A} \mr{F_l} Hom_{op \gamma}(\B,\C)$,  
$\cc{B} \mr{F_r} Hom_{\gamma}(\A,\C)$ be $2$ functors in correspondence as in  \mbox{\S$\,$\ref{sub:terminology}, item \ref{cartesianHomp}.}
Consider weights $\A \mr{W_l} \Cat$, $\B \mr{W_r} \Cat$. 
Then, the following holds: 
$$
\{W_l,\, \{W_r,\, F_r\}_{\beta} \}_\alpha \; \cong \; 
\{W_r,\, \{W_l,\, F_l\}_\alpha \}_{\beta} $$
\end{proposition}
\begin{proof}
By the usual reasoning it suffices to show it for the case $\C = \Cat$.
We have the following isomorphisms given by Proposition \ref{cathas} and Corollary \ref{comeout}:
$$
\{W_l,\, \{W_r,\, F_r\}_{\beta} \}_\alpha  \cong 
Hom_{\alpha}(\A,\Cat)(W_l,\, \{W_r,\, F_r\}_{\beta}) 
\cong \{W_r,\, Hom_{\alpha}(\A,\Cat)(W_l,\, F_r-) \}_{\beta}
$$
We conclude the proof by showing that 
$\{W_l,\, F_l\}_\alpha =  Hom_{\alpha}(\A,\Cat)(W_l,\, F_r-)$. By Proposition  \ref{prop:claxcolimptoapto} we can compute pointwise:

\hfill
$
\{W_l,\, F_l\}_\alpha(B) \;=\;
\{W_l,\, (F_l-)(B)\}_\alpha \;=\;
Hom_{\alpha}(\A,\Cat)(W_l,\, (F_l-)(B)) \;=\;
$
$
Hom_{\alpha}(\A,\Cat)(W_l,\, F_rB(-)) =
Hom_{\alpha}(\A,\Cat)(W_l,\, F_r-)(B).
$

\vspace{1ex}

The second equality is justified  by Proposition \ref{cathas}, the third one follows from the formulas in 
\S$\,$\ref{sub:terminology}, item \ref{cartesianHomp}, and the last is clear.
\end{proof}

It is convenient to state with a slightly different notation a particular case which we will need in this paper;
 
\begin{proposition} \label{sin:interchange} 
Let $\sigma \in \Lbl_\cc{I}$. Consider a 
$2$-functor 
$\cc{I} \xr{F_{(-)}} Hom_p(\A,\B)$
and a weight 
$\A \mr{W} \Cat$. Then the following holds: 
$$ 
\{W,\Lim{i \in \cc{I}}{F_i}\}_p = \Lim{i \in \cc{I}}{\{W,F_i\}_p}.
$$

\vspace{-5ex}

\cqd 
\end{proposition}

\vspace{1ex}

The commutativity of weighted pseudolimits with conical $\sigma$-colimits, which is (as usual) a much deeper subject, is treated in \cite{Otropaper} for $\Cat$-valued $2$-functors. We recall now this result noting that, while in \cite[Theorem 3.2]{Otropaper} it is stated for a $2$-functor \mbox{$\cc{I} \mr{} \cc{H}om_s(\A,\Cat)$,} since pseudolimits are also computed pointwise in $\cc{H}om_p(\A,\Cat)$  a careful inspection of the proof yields:

\begin{theorem} \label{teo:conmutan}
Let $\cc{I}$ be a \s-filtered $2$-category, $\A$ a $2$-category. Consider a 
$2$-functor 
$\cc{I} \xr{F_{(-)}} Hom_p(\A,\Cat)$
and a finite weight (see Definition  \ref{def:leftexact})  
$\A \mr{W} \Cat$. Then 
the canonical comparison functor 
$$
\coLim{i \in \cc{I}}{ \;{}_{bi} \!\{ W,F_i \} }_p \mr{\approx} \;{}_{bi} \!\{ W,\coLim{i \in \cc{I}}{F_i} \}_p 
$$
\noindent is an equivalence of categories. \qed
\end{theorem}

\section{$\sigma$-filtered $2$-categories}\label{sigma2f}

We fix throughout this section a $2$-category $\cc{C}$ and a 
$1$-subcategory $\Sigma$ of $\cc{C}$ which contains all the objects. Note that this amounts to a family, that we will also denote by $\Sigma$, of arrows of $\cc{C}$
such that all the identities belong to $\Sigma$ and $\Sigma$ is closed by composition. We don't require $\Sigma$ to contain the isomorphisms or the equivalences of $\cc{C}$.

\subsection{The notion of \s-filtered} \label{sub:sigmafiltered}

 Recall that a non empty $1$-category is filtered if and only if every finite diagram has a cone (see \cite[\S$\,$VII.6]{MMcL}). This happens if and only if two particular diagrams, corresponding to binary products and equalizers, have a cone (the two usual axioms of filtered category). 
 
 In the $2$-dimensional case the notion of filteredness has been considered under the name of \emph{bifiltered} in \cite{K}, and \emph{$2$-filtered} in \cite{DS}, and it holds if and only if every finite diagram has a  \emph{pseudocone} (see \cite{Data}).  
 
 We introduce now the concept of $2$-filteredness with respect to a family $\Sigma$. It generalizes the definition of bifiltered in \cite{K}, which corresponds to the case where $\Sigma$ consists of all the arrows of $\A$.

\begin{notation}
We add a circle to an arrow $\xymatrix{\cdot \ar[r]|{o} & \cdot}$ to indicate that it belongs to $\Sigma$. 
\end{notation}

\begin{\de}
\label{def:filtered}
 We say that a pair $(\C, \; \Sigma)$ is $\sigma$-filtered, or for brevity, that $\cc{C}$ is \s-filtered (with respect to $\Sigma$), if it is non empty and the following hold:
 
 \bigskip 
  
 \noindent
 \begin{tabular}{rl}
  ${\sigma {\bf F0}}.$ & Given $A, B \in \C$, there exist $E \in \C$ and morphisms $\vcenter{\xymatrix@R=0pc{A \ar[rd]|{o}^f \\
                                                                                                   & E. \\
                                                                                                   B \ar[ru]|{o}_g}}$ \\
  ${\sigma {\bf F1}}.$ & Given $\xymatrix{A \ar@<1ex>[r]^f \ar@<-1ex>[r]|{o}_g & B} \in \C$, there exist a morphism $\xymatrix{B \ar[r]|{o}^h & E}$ and a 2-cell $hf \Mr{\alpha} hg$. \\
  & If $f \in \Sigma$, we may choose $\alpha$ invertible. \\
  ${\sigma {\bf F2}}.$ & Given $A \cellpairrdc{f}{\alpha}{\beta}{g} B \in \C$, there exists a morphism $\xymatrix{B \ar[r]|{o}^h & E}$ such that $h \alpha = h\beta $.
 \end{tabular}
  
  \bigskip 
  
  We say that $\C$ is $\sigma$-cofiltered if $\C^{op}$ is $\sigma$-filtered. We keep the same labels for the axioms. 
\end{\de}

\begin{\prop} \label{prop:notableclimits}
 Consider the following finite diagrams:
  
  \begin{center}
  \begin{tabular}{ll}
   1. $\{a , b\} \mr{F_1} \C$, & $\{C , D\} $ \\
  
  2. $\{a \mrpair{u}{v}b\} \mr{F_2} \C$, & $\{C \mrpairc{f}{g} D\} $ \\
  
  3. $\{a \cellpairrd{u}{\theta}{\eta}{v} b \}  \mr{F_3} \C$, & $\{C \cellpairrdc{f}{\alpha}{\beta}{g} D\}$    
  \end{tabular}
  \end{center}
  
    Let $\Sigma_1$, $\Sigma_2$, $\Sigma_3$ (respectively) be the family of arrows that are mapped to arrows of $\Sigma$, i.e. $\Sigma_i = F_i^{-1}(\Sigma)$.
  Then, for each $i$, the category $Cones_\sigma^{\Sigma_i}(F_i,E)$ (recall Definition \ref{def:claxlimit}) 
  is equivalent (naturally in $E$) to the category $\A_i$ whose objects and arrows are:
  
  \begin{enumerate}
   \item[$\A_1$]  \begin{enumerate}
    \item[]Objects: Pairs of morphisms $\vcenter{\xymatrix@R=0pc{C \ar[rd]^h \\
                                                                                                   & E \\
                                                                                                   D \ar[ru]_\ell}}$
    \item[]Arrows: Pairs of $2$-cells $h \Mr{} h'$, $\ell \Mr{} \ell'$.
   \end{enumerate}

  \item[$\A_2$] \begin{enumerate}
    \item[]Objects: An object consists of a morphism $\xymatrix{D \ar[r]^h & E}$ together with a 2-cell \mbox{$hf \Mr{\gamma} hg$,} invertible if 
  $f \in \Sigma$.
    \item[]Arrows: $2$-cells $h \Mr{\eta} h'$ such that $\gamma'(\eta f) = (\eta g) \gamma$.
   \end{enumerate}
  
  \item[$\A_3$] \begin{enumerate}
    \item[]Objects: Morphisms $\xymatrix{D \ar[r]^h & E}$ such that $h \alpha = h\beta $.
    \item[]Arrows: $2$-cells $h \Mr{\eta} h'$.
   \end{enumerate}
  \end{enumerate} 
\end{\prop}

\begin{proof}
Certainly item 1 requires no proof. 

For item 2, we define the equivalence $Cones_\sigma^{\Sigma_2}(F_2,E) \mr{\phi} \A_2$, and leave the verification of the details to the reader. Given a $\sigma$-cone $\theta$, define $h = \theta_b$, $\gamma = \theta_v^{-1} \theta_u$. Given a morphism of $\sigma$-cones $\theta \mr{\varphi} \theta'$, define $\eta = \varphi_b$. Then it is easy to check that $\phi$ is actually surjective on objects, and given $\phi(\theta) \mr{\eta} \phi(\theta')$, the unique $\theta \mr{\varphi} \theta'$ such that $\phi(\varphi) = \eta$ is defined by $\varphi_b = \eta$, $\varphi_a = \theta_v' (\eta g) \theta_v^{-1}$.

For item 3, define $h, \gamma$ and $\eta$ as in item 2, but now since $\theta$ is a $\sigma$-cone we have 
$h\alpha = \gamma = h\beta$. Then in this case the condition $\gamma'(\eta f) = (\eta g) \gamma$ for a $2$-cell $h \Mr{\eta} h'$ is $(h' \alpha)(\eta f) = (\eta g)(h \alpha)$, which holds by the interchange law.
\end{proof}

The following proposition expresses a basic property of $\sigma$-filteredness and it is a \mbox{generalization} of \cite[\S$\,$VII.6, Lemma 1]{MMcL} to the $2$-dimensional case. See also \cite{Data} where the case of $\Sigma$ consisting of all the arrows of $\C$ is analyzed. 

\begin{notation}
For a $2$-functor $\Delta \mr{F} \C$, we say that a \s-cone 
$\theta$ with vertex $E$ has \emph{arrows in $\Sigma$} if the structural arrows $F(i) \mr{\theta_i} E$ are in $\Sigma$ for all $i \in \Delta$.
\end{notation}

\begin{proposition} \label{prop:filtsiicone}
The following are equivalent
\begin{enumerate}
 \item[i)] $\C$ is $\sigma$-filtered.
 \item[ii)] Each of the diagrams $F_1,F_2,F_3$ in Proposition \ref{prop:notableclimits} has a $\sigma$-cone (with respect to $F^{-1}(\Sigma)$) with arrows in $\Sigma$.
 \item[iii)] Every finite $2$-diagram $\Delta \mr{F} \C$ (i.e. every $2$-functor $\Delta \mr{F} \C$ with $\Delta$ a finite $2$-category) has a $\sigma$-cone (with respect to $F^{-1}(\Sigma)$) with arrows in $\Sigma$. 
\end{enumerate}

\end{proposition}

\begin{proof}
$iii) \Rightarrow ii)$ is trivial. $ii) \Rightarrow i)$ follows from the description of the $\sigma$-cones in Proposition \ref{prop:notableclimits}. To show $i) \Rightarrow iii)$, 
suppose that $\C$ is $\sigma$-filtered and let $\Delta \mr{F} \C$ be a finite $2$-diagram. Since $\Delta$ is finite, by  axiom $\sigma {\bf F0}$, we have morphisms $\big\{\xymatrix{Fi \ar[r]|{o}^{\theta_i} & E }\big\}_{i\in \Delta}$. 

\vspace{1ex}

We will modify $E$ and the arrows $\theta_i$ by going further, in order to have a $\sigma$-cone with arrows in $\Sigma$. We will do this one arrow $u$ of $\Delta$ at a time. Using axiom $\sigma {\bf F1}$, there is a morphism $\xymatrix{E \ar[r]|{o}^{h} & E'}$ and a $2$-cell $\theta_j F(u) \Mr{\theta_u} \theta_i$, invertible if $F(u) \in \Sigma$. We denote $E'$ by $E$ again, the compositions
$h \theta_i$ by $\theta_i$, and $h \theta_u$ by $\theta_u$ for all the pre-existing $\theta_u$. We repeat the procedure to have $\left\{ \theta_j F(u) \Mr{\theta_u} \theta_i \right\}_{i\mr{u} j\in \Delta}$, with $\theta_u$ invertible for all $u$ such that $F(u) \in \Sigma$.

Now we consider the equations {\bf LN0}, {\bf LN1}, {\bf LN2} of \S$\,$\ref{sub:terminology}, item \ref{laxoplaxnatural} expressing the lax naturality of $F \Mr{\theta} k_E$ (see Definition \ref{def:claxlimit}). A similar procedure, considering one equation at a time and using axiom $\sigma {\bf F2}$ instead of $\sigma {\bf F1}$, allows one to go further and make $\theta$ a $\sigma$-cone with arrows in $\Sigma$.
\end{proof}

\begin{remark}
The reason why we consider $\sigma$-cones with arrows in $\Sigma$ 
will be clear in Proposition \ref{prop:bilimitsselevantan} below. 
We note nevertheless that a \emph{weaker} notion of $\sigma$-filteredness where we 
ask that every finite 2-diagram has a $\sigma$-cone could also be worth considering in another context.
\end{remark}

\subsection{Exact $2$-functors} \label{sub:exact}

\begin{\de} \label{de:conmutarconlim} 
 Consider $2$-functors  $\A \mr{W} \Cat$, $\A \mr{F} \B \mr{H} \C$. 
  We say that $H$ preserves a $\sigma$-bilimit ${}_{bi} \! \{W,F\}_{\sigma}$ if $\,{}_{bi} \! \{W,HF\}_{\sigma}$ exists, and the canonical map \mbox{$H {}_{bi} \! \{W,F\}_{\sigma} \mr{} {}_{bi} \! \{W,HF\}_{\sigma}$} is an equivalence.
\end{\de}
We will define the notion of finite weight and finite bilimit below. We note that there is a more general notion of finite (or finitary) weight (\cite[\S$\,$4]{K3},\cite{S}) which we don't consider here as it is not necessary for our purposes.
\begin{\de} \label{def:leftexact} $ $

$1.$ We say that a $2$-functor $\A \mr{W} \Cat$ is a \emph{finite weight} if $\A$ is a finite $2$-category and for each $A \in \A$, $WA$ is a finite category. A \emph{finite bilimit} is a weighted bilimit with finite weight. 

$2.$ Assume that $\B$ is a $2$-category with finite weighted bilimits. We say that a $2$-functor $\B \mr{H} \C$ is \emph{left exact} if it preserves all finite bilimits.

\end{\de}
Note that all finite weighted bilimits are required to exist in the domain category of exact $2$-functors, but not necessarily in the codomain category.

The objective of this subsection is to prove the following result: for any left exact $\Cat$-valued $2$-functor $P$, its 2-category of elements $\cc{E}l_P$ is $\sigma$-cofiltered (with respect to the cocartesian arrows). The first result which we will use is a $2$-dimensional version of a result that is known for $\cc{S}et$-valued functors, see for example \cite[Proposition 4.87]{K1}:
\begin{\prop} \label{prop:bilimitsselevantan}
 Let $\A \mr{P} \Cat$ be a $2$-functor. Let $\Delta \mr{F} \cc{E}l_P$ be a $2$-functor, and set $\Sigma = F^{-1}(\cart_P)$. Assume that $\Diamond_P F$ has a $\sigma$-bilimit $L$ in $\A$ that is preserved by $P$. Then there exist $c \in PL$ and a $\sigma$-cone for $F$ with arrows in $\cart_P$ with vertex $(c,L)$, which is the $\sigma$-bilimit of $F$.
\end{\prop}
\begin{proof}
We are going to denote the action of $F$ by $Fi=(a_i,F_i)$ for each $i\in \Delta$, $Fu=(F_u,\sigma_u)$ for each $i \mr{u} j \in \Delta$, $F\theta=F_\theta$ for each $i \cellrd{u}{\theta}{v} j \in \Delta$.

Consider the $\sigma$-bilimit $L$ of the $2$-functor $\Delta \xr{\Diamond_P F} \A$, then $L$ is furnished with a  
$\sigma$-cone $\{L \xr{h_i} F_i\}_{i\in \Delta}$, $\{F_u h_i \Mr{h_u} h_j\}_{i\mr{u}j\in \Delta}$.
This $\sigma$-bilimit is preserved by $P$, this means that if we denote by $E$ the $\sigma$-bilimit of the $2$-functor $\Delta \xr{P \Diamond_P F} \Cat$, which is furnished with a $\sigma$-cone $\{E \mr{\pi_i} PF_i\}_{i\in \Delta}$, $\{PF_u \pi_i \Mr{\pi_u} \pi_j\}_{i\mr{u} j \in \Delta}$, then the comparison functor $PL \mr{s} E$ such that $\pi_i s= Ph_i$, $\pi_u s=Ph_u$ is an equivalence of categories.

Recall that, by the construction of $\sigma$-limits in $\Cat$ given in Remark \ref{rem:limitsenCat}, 
we have that $E=\cc{H}om_{\sigma}(\Delta,\Cat)(k_1,P\Diamond_P F)$ and so:

\begin{enumerate}
 \item Objects of $E$ are $\sigma$-natural transformations between the constant $2$-functor $k_1$ and $P \Diamond_P F$. Observe that those transformations correspond to pairs of tuples $(\{x_i\}_{i\in \Delta}, \{\varphi_u\}_{i\mr{u}j \in \Delta})$ with $x_i \in PF_i$, $PF_u(x_i) \mr{\varphi_u} x_j$ satisfying the following properties corresponding to axioms {\bf LN0}, {\bf LN1} and {\bf LN2} from \S$\,$\ref{sub:terminology}, item \ref{laxoplaxnatural}:
 
\begin{tabular}{ll}
 {\bf LN0}. For all $i\in \Delta$, & $\varphi_{id_i} = id_{x_i}$ \\
 {\bf LN1}. For all $i \mr{u} j \mr{v} k \in \Delta$, & $\varphi_{vu} = \varphi_v  PF_v(\varphi_u)$\\
 {\bf LN2}. For all $i \cellrd{u}{\theta}{v} j \in \Delta$, & $ {\varphi}_{v} (PF_{\theta})_{x_i} = \varphi_u$ 
\end{tabular}

 \item Arrows of $E$ are modifications. Observe that a modification between ($\{x_i\}_{i\in \Delta}, \{\varphi_u\}_{i\mr{u}j \in \Delta})$ and $(\{y_i\}_{i\in \Delta}, \{\psi_u\}_{i\mr{u}j \in \Delta})$ corresponds to a tuple $\{x_i \mr{\xi_i} y_i\}_{i\in \Delta}$ satisfying the property corresponding to axiom {\bf LNM} from \S$\,$\ref{sub:terminology}, item \ref{laxoplaxnatural}:
 
\begin{tabular}{ll}
{\bf LNM}. For all $i \mr{u} j \in \Delta$, & $\psi_{u} PF_u(\xi_i)= \xi_j \varphi_u$
\end{tabular}

 \item $\pi_i(\{x_i\}_{i\in \Delta}, \{\varphi_u\}_{i\mr{u}j \in \Delta})=x_i$, $\pi_i(\{x_i \mr{\xi_i} y_i\}_{i\in \Delta})=\xi_i$.
 
 \item $(\pi_u)_{(\{x_i\}_{i\in \Delta}, \{\varphi_u\}_{i\mr{u}j \in \Delta})}=\varphi_u$.
\end{enumerate}

Thus $s(c)=(\{Ph_i(c)\}_{i\in \Delta}, \{(Ph_u)_c\}_{i\mr{u}j \in \Delta})$.

\vspace{1ex}

Now, $F$ determines an object of $E$ $(\{a_i\}_{i\in \Delta}, \{\sigma_u\}_{i\mr{u}j \in \Delta})$. Since $s$ is an equivalence of categories, there exists an object $c\in PL$ such that we have an invertible modification from $s(c)$ to $(\{a_i\}_{i\in \Delta}, \{\sigma_u\}_{i\mr{u}j \in \Delta})$, say $\{Ph_i(c) \mr{\xi_i} a_i \}_{i\in \Delta}$ satisfying that the following diagram commutes in $PF_j$:

$$
\xymatrix@C=2cm{P(F_uh_i)(c)\ar[d]_{(Ph_u)_c} \ar[r]^{PF_u(\xi_i)} & PF_u(a_i)\ar[d]^{\sigma_u} \\
          Ph_j(c)\ar[r]_{\xi_j} & a_j } 
$$
            
We have the following $\sigma$-cone for $F$:

$$\{(c,L) \xr{(h_i,\xi_i)} (a_i,F_i)\}_{i\in \Delta}, \quad \{(F_u h_i,\sigma_u PF_u(\xi_i)) \Mr{h_u} (h_j,\xi_j)\}_{i\mr{u}j\in \Delta}.$$

It is straightforward from the diagram above that the  $h_u$ are $2$-cells in $\cc{E}l_P$. 

\vspace{1ex}

We leave the verification of the fact that $(c,L)$ is actually a \mbox{$\sigma$-bilimit} to the interested reader. In any case, note that in this paper we only need the existence of a $\sigma$-cone (with arrows in $\cart_P$).
\end{proof}

We write here, for convenience of the reader, the {\em dual} version of Proposition \ref{prop:notableclimits}:

\begin{proposition} \label{prop:dualdenotableclimits}
For the 2-functors $F_i$, $i=1,2,3$ considered in Proposition \ref{prop:notableclimits}, and $\Sigma_i = F_i^{-1}(\Sigma)$, 
  the category $Cones_\sigma^{\Sigma_i}(E,F_i)$ (recall Remark \ref{rem:notgray}) 
  is equivalent (naturally in $E$) to the category $\cc{B}_i$ whose objects and arrows are:
  
  \begin{enumerate}
   \item[$\cc{B}_1$]  \begin{enumerate}
    \item[]Objects: Pairs of morphisms $\vcenter{\xymatrix@R=0pc{& C \\
                                                                                                    E \ar[ru]^h \ar[rd]_\ell \\
                                                                                                   & D }}$
    \item[]Arrows: Pairs of $2$-cells $h \Mr{} h'$, $\ell \Mr{} \ell'$.
   \end{enumerate}

  \item[$\cc{B}_2$] \begin{enumerate}
    \item[]Objects: An object consists of a morphism $\xymatrix{E \ar[r]^h & C}$ together with a 2-cell \mbox{$fh \Mr{\gamma} gh$,} invertible if 
  $f \in \Sigma$.
    \item[]Arrows: $2$-cells $h \Mr{\eta} h'$ such that $\gamma'(f \eta) = (g \eta) \gamma$.
   \end{enumerate}
  
  \item[$\cc{B}_3$] \begin{enumerate}
    \item[]Objects: Morphisms $\xymatrix{E \ar[r]^h & C}$ such that $\alpha h = \beta h$.
    \item[]Arrows: $2$-cells $h \Mr{\eta} h'$. \qed
   \end{enumerate}
  \end{enumerate}

\end{proposition}

\begin{remark} \label{rem:notableweightedbilimits}
Concerning our objective of showing that $\cc{E}l_P$ is $\sigma$-cofiltered for a left exact 2-functor 
$\cc{A} \mr{P} \Cat$, recall Proposition  \ref{prop:filtsiicone}. For a $2$-functor $\Delta \mr{F} \cc{E}l_P$, in view of Proposition \ref{prop:bilimitsselevantan}, we can deduce that $F$ has a $\sigma$-cone with arrows in $\cart_P$ by showing that the $\sigma$-bilimit of the composite 2-functor $\Diamond_P F$ is preserved by $P$. We will show that when $P$ is exact, this is the case when $F$ is each of the functors $F_1,F_2,F_3$ considered in Proposition \ref{prop:notableclimits}, by relating the $\sigma$-bilimits of these functors to biproducts, biinserters, biequalizers and biequifiers in $\cc{A}$ (which we describe below). This is done by performing a careful comparison of the categories $\cc{B}_1, \cc{B}_2, \cc{B}_3$ above with the cones of these four bilimits. Consider thus the finite diagrams $\Delta \mr{F_i} \C$, $i=1,2,3$:

\smallskip 

\noindent 1. $biproduct(C,D)$: this is the bilimit of the diagram $a \stackrel{F_1}{\longmapsto} C,  b \stackrel{F_1}{\longmapsto} D$ weighted by the 2-functor $W_1$ constant at the terminal category $1$.

\noindent 2. $biinserter(f,g)$: this is the bilimit of the diagram $a \mrpair{u}{v} b \stackrel{F_2}{\longmapsto} C \mrpair{f}{g} D$ weighted by $a \mrpair{u}{v}b \stackrel{W_2}{\longmapsto} 1 \mrpair{0}{1} 2$, see \cite[(4.1)]{K2} for details. Note that if we consider the category $I$ (consisting of two objects and an isomorphism) instead of $2$, the weighted bilimit is the $biequalizer(f,g)$ (note that the $biequalizer(f,g)$ is also the $biisoinserter(f,g)$ see \cite[p.308]{K2} and \cite[Observation 5.23]{Canevali})

\noindent 3. $biequifier(\alpha,\beta)$: this is the bilimit of the diagram $a \cellpairrd{u}{\theta}{\eta}{v} b  \stackrel{F_3}{\longmapsto} C \cellpairrd{f}{\alpha}{\beta}{g} D$ weighted by $a \cellpairrd{u}{\theta}{\eta}{v} b  \stackrel{W_3}{\longmapsto} 1 \cellpairrd{0}{}{}{1} 2$, see \mbox{\cite[(4.5)]{K2}} for details. 

We note, though we won't use this fact, that finite biproducts, biequalizers and bicotensor products with $2= \{0 \rightarrow 1\}$ suffice to construct all finite weighted bilimits (The general proof in \cite{SCorr} can be restricted to the finite case, see \cite[\S$\,$6.2]{Canevali} for details). Since the bicotensor product with $2$ can be constructed from the biinserter and the biequifier, the four bilimits above are also sufficient to construct all finite weighted bilimits. In particular we have that a $2$-functor $H$ is left exact if and only if these four bilimits exist and are preserved by $H$, and thus we are actually using in our proof of Proposition \ref{prop:exactothenfiltering} the full strength of the hypothesis. 
\end{remark}

\begin{proposition} \label{prop:descripcionWconosfinitos}
With the definitions above, for $i=1,2,3$, the category $Cones_p^{W_i}(E,F_i)$ (recall Definition \ref{de:NE}) 
  is equivalent (naturally in $E$) to the category $\cc{B}_i$ described in Proposition \ref{prop:dualdenotableclimits} (for $i=2$, the case in which the 2-cell $\gamma$ is required to be invertible corresponds to the case of the biequalizer). 
\end{proposition}

\begin{proof}
The case $i=1$ requires no proof. For the case $i = 2$, we denote the category $2$ by \mbox{$\{0 \mr{\ell} 1\}$. }
 Note that a cone $W_2 \Mr{\theta} \C(E,F_2(-))$ amounts to $E \mr{\theta_a} C$, $E \cellrd{\theta_b(0)}{\theta_b(\ell)}{\theta_b(1)} D$, and invertible $2$-cells $\theta_a f \Mr{\theta_u} \theta_b(0)$, $\theta_a g \Mr{\theta_v} \theta_b(1)$. The definition of the equivalence $Cones_p^{W_2}(F_2,E) \mr{\phi} \cc{B}_2$ on objects is by the formulas $h = \theta_a$, $\gamma = \theta_v^{-1} \theta_b(\ell) \theta_u$, this is easily seen to be surjective.
 
 A morphism of cones $W_2 \cellrd{\theta}{\varphi}{\theta'} \C(E,F_2(-))$ is a modification given by natural transformations $\theta_a \Mr{\varphi_a} \theta'_a$, $\theta_b \Mr{\varphi_b} \theta'_b$, therefore by $2$-cells $\varphi_a$, $(\varphi_b)_0$, $(\varphi_b)_1$ such that \mbox{$\theta_b'(\ell) (\varphi_b)_0 = (\varphi_b)_1 \theta_b(\ell)$,}   $\theta_u' (f \varphi_a) \theta_u^{-1} = (\varphi_b)_0$, $\theta_v' (g \varphi_a) \theta_v^{-1} = (\varphi_b)_1$. The definition of $\phi$ is by the formula $\eta = \varphi_a$, then from the equations above we note that $(\varphi_b)_0$ and $(\varphi_b)_1$ are determined by $\varphi_a$, and the condition $\gamma' (f \eta) = (g \eta) \gamma$ is equivalent to the equation $\theta_b'(\ell) (\varphi_b)_0 = (\varphi_b)_1 \theta_b(\ell)$. Then $\varphi$ is full and faithful.
 
 In the case where we replace $2$ by $I$, the formulas are the same, simply note that $\gamma$ is invertible. If $i = 3$, we consider $h, \gamma, \eta$ as in the case $i=2$, then from the $2$-naturality of $\theta$ it follows $\alpha h = \gamma = \beta h$ and the proof finishes like the proof of Proposition \ref{prop:notableclimits}.
\end{proof}

\begin{remark} \label{rem:PexactosiipreservalosFi}
From Proposition \ref{prop:descripcionWconosfinitos} it follows that the $\sigma$-bilimit of each of the functors $F_1,F_2,F_3$ of Proposition \ref{prop:notableclimits}, for any of the possibilities for the family $\Sigma_i$, is a finite weighted bilimit, more precisely:
\begin{enumerate}
 \item For $\{a , b\} \stackrel{F_1}{\longmapsto} \{C,D\}$, $\quad \cbiLim {}{F_1} = biproduct(C,D)$.
 \item  
 \begin{enumerate}
  \item[(i)] For $\{a \mrpaircc{u}{v}b\} \stackrel{F_2}{\longmapsto} \{C \mrpair{f}{g}D\}$, $\quad \cbiLim {}{F_2} = biequalizer(f,g)$.
  \item[(ii)] For $\{a \mrpairc{u}{v}b\} \stackrel{F_2}{\longmapsto} \{C \mrpair{f}{g}D\}$, $\quad \cbiLim {}{F_2} = biinserter(f,g)$. 
 \end{enumerate}
 \item For $\{a \cellpairrdcc{u}{\theta}{\eta}{v}b\} \stackrel{F_3}{\longmapsto} \{C \cellpairrd{f}{\alpha}{\beta}{g}D\} $, and for \\
 \textcolor{white}{For} $\{a \cellpairrdc{u}{\theta}{\eta}{v}b\} \stackrel{F_3}{\longmapsto} \{C \cellpairrd{f}{\alpha}{\beta}{g}D\} $,
 $\quad \cbiLim {}{F_3} = biequifier(\alpha,\beta)$. 
\end{enumerate}

\vspace{1ex}

It follows that if a $2$-functor $P$ is left exact then $P$ preserves the $\cbiLim{}{F_i}$, for $i=1,2,3$ with the $1$-subcategories $\Sigma_i$ considered above.
\end{remark}

\begin{\prop} \label{prop:exactothenfiltering}
Let $\A$ be a $2$-category with finite weighted bilimits and let $\A \mr{P} \Cat$ be a $2$-functor. If $P$ is left exact then $\cc{E}l_P$ is $\sigma$-cofiltered with respect to $\cart_P$.
\end{\prop}
\begin{proof}
By Proposition \ref{prop:filtsiicone} it suffices to show that 
 each of the diagrams 
 $F_1,F_2,F_3: \Delta \mr{} \cc{E}l_P$ 
 considered in Proposition \ref{prop:notableclimits} has a $\sigma$-cone with arrows in $\cart_P$. 
 Let $i=1,2,3$, 
 by Remark \ref{rem:PexactosiipreservalosFi} $\cbiLim{}{\Diamond_P F_i}$ exists in $\A$ and is preserved by $P$ (note that all the possible $\Sigma_i = F_i^{-1} (\cart_P)$ were considered in the remark). 
 Then by Proposition \ref{prop:bilimitsselevantan} $F_i$ has a $\sigma$-cone with arrows in $\cart_P$ which concludes the proof.
 \end{proof}

\subsection{$\sigma$-cofinal $2$-functors} \label{sub:cofinal}

In this section we define \emph{\s-cofinal} $2$-functors and establish some properties that will be used in the proof of Theorem \ref{th:main}. Our definition is a $2$-dimensional $\sigma$-version of the definition in SGA4 for the case when $\C$ is filtered (see \cite[8.1.1]{G2}). If $\Sigma = \C_0$, we recover \mbox{\cite[Definition 1.3.1]{tesisEmi}.} We do not deal with a more general concept of \s-cofinality since this particular case is relevant enough and it is the only one that we need in this paper. We leave the development of the full theory of $\sigma$-cofinal $2$-functors for future work.

\begin{\de} \label{def:initial} 
 Let $\C$, $\C'$ be $2$-categories and $\Sigma$, $\Sigma'$ $1$-subcategories of $\C$, $\C'$ respectively. Suppose that $\C$ is $\sigma$-filtered. We say that a $2$-functor $\C \mr{T} \C'$ is $\sigma$-cofinal (with respect to $\Sigma$ and $\Sigma'$) if it satisfies:

 \bigskip
 
 \begin{tabular}{r p{13cm}}
  $\sigma {\bf C0}.$ & Given $C'\in \C'$, there exist $C \in \C$ and a morphism $\xymatrix{C' \ar[r]|{o} & TC}$ in $\C'$. \\
  $\sigma {\bf C1}.$ & Given $\xymatrix{ C' \ar@<1ex>[r]^f \ar@<-1ex>[r]|{o}_g & TC} \in \C'$, there exist a morphism $\xymatrix{C \ar[r]|{o}^u & D}$ and a 2-cell $T(u)f \Mr{\alpha} T(u)g$. If $f \in \Sigma'$, we may choose $\alpha$ invertible. \\
  $\sigma {\bf C2}.$ & Given $C \in \C$, $C' \in \C'$ and 2-cells $C' \cellpairrdc{f}{\alpha}{\beta}{g} TC \in \C'$, there exists a morphism $\xymatrix{C \ar[r]|{o}^{u} & D} \in \C$ such that $T(u) \alpha = T(u) \beta$.
 \end{tabular}

 \bigskip
 
 If $\C$ is $\sigma$-cofiltered, we say that $\C \mr{T} \C'$ is $\sigma$-initial if $\C^{op} \mr{T^{op}} \C'^{op}$ is $\sigma$-cofinal. We keep the same labels for the axioms.
 \end{\de}
 
The following proposition is the only result concerning $\sigma$-cofinal functors that we need in this paper, and it 
is the analogous to item c) of \cite[8.1.3]{G2}.

\begin{\prop} \label{prop:initialencofiltered}
 Let $\C$, $\C'$ be $2$-categories, $\C \mr{T} \C'$ a $2$-functor, $\Sigma'$ a subcategory of $\C'$, and $\Sigma = T^{-1}(\Sigma')$.
  If the following hold:

\vspace{1ex}
$\hspace{2ex} 1.\;$ $\C'$ is $\sigma$-filtered,

$\hspace{2ex} 2.\;$ $T$ is pseudo-fully-faithful,

$\hspace{2ex} 3.\;$  Condition $\sigma {\bf C0}$ from Definition \ref{def:initial}.

\vspace{1ex}

\noindent Then $\C$ is $\sigma$-filtered and $T$ is $\sigma$-cofinal. 
\end{\prop}

\begin{proof}
 We observe that since $T$ is pseudo-fully-faithful and $\Sigma = T^{-1}(\Sigma')$ we have:

 \bigskip
 
 \noindent
 {\bf (1)} For every arrow $\xymatrix{TC \ar[r]|{o}^{h} & TD} \hbox{ in } \C'$, there exists $\xymatrix{C \ar[r]|{o}^{u} & D}$ such that $T(u)\cong h$.
 
 \bigskip

We are going to check first that axioms $\sigma {\bf C1}$ and $\sigma {\bf C2}$ from Definition \ref{def:initial} are satisfied: 

\vspace{1ex}

\noindent
$\sigma {\bf C1}.$ Given $C \in \C$, $C' \in \C'$, $C' \mrpairc{f}{g} TC  \in \C'$, since $\C'$ is $\sigma$-filtered, there exist a morphism $\xymatrix{TC \ar[r]|{o}^{h} & D'}$ and a 2-cell $hf\Mr{\alpha} hg$, that we may take invertible if $f \in \Sigma'$. Then, by the fact that condition $\sigma {\bf C0}$ is satisfied, there exist an object $D \in \C$ and a morphism \mbox{$\xymatrix{D' \ar[r]|{o}^{l} & TD} \in \C'$.} Now, by {\bf (1)} above, there exists a morphism $\xymatrix{C \ar[r]|{o}^{u} & D}$ such that $T(u)\cong lh$ and so we have a 2-cell $T(u)f \cong lhf \Mr{l \alpha} lhg \cong T(u)g$, which is invertible if $f \in \Sigma'$.

\noindent
$\sigma {\bf C2}.$ Given $C \in \C$, $C' \in \C'$ and 2-cells $C' \cellpairrdc{f}{\alpha}{\beta}{g} TC \in \C'$, since $\C'$ is $\sigma$-filtered, there exists a morphism $\xymatrix{TC \ar[r]|{o}^{h} & D'}$ such that $h \alpha = h\beta$. Then, by the fact that condition $\sigma {\bf C0}$ is satisfied, there exist an object $D\in \C$ and a morphism $\xymatrix{D' \ar[r]|{o}^{l} & TD}$. Now, by {\bf (1)} above, there exists a morphism $\xymatrix{C \ar[r]|{o}^{u} & D}$ such that $T(u)\cong lh$. This, together with the fact that $h \alpha = h \beta$ can be used to prove that 
  $T(u) \alpha = T(u) \beta $.

\vspace{1ex}

It only remains to check that $\sigma {\bf F0}$, $\sigma {\bf F1}$ and $\sigma {\bf F2}$ from Definition \ref{def:filtered} are satisfied for the $2$-category $\C$:

\noindent
$\sigma {\bf F0}.$ Given $C, D\in \C$, since $\C'$ is $\sigma$-filtered, there exist morphisms $\vcenter{\xymatrix@R=0pc{ TC \ar[rd]|{o}^f \\
                                                                                                                                  & E' \\
                                                                                                                                   TD \ar[ru]|{o}_g }}$. 
Then, since $T$ is $\sigma$-cofinal, there exist an object $E\in \C$ and a morphism $\xymatrix{E' \ar[r]|{o}^{h} & TE}$. Now, by {\bf (1)} above, this yields morphisms $\vcenter{\xymatrix@R=0pc{ C \ar[rd]|{o}^u   \\
                                              & E \\
                                                D \ar[ru]|{o}_v}}$.   

\noindent
$\sigma {\bf F1}.$ Given $C \mrpairc{u}{v} D \in \C$, consider $TC \mrpairc{Tu}{Tv} TD \in \C'$. Since $T$ is $\sigma$-cofinal, there exist a morphism  $\xymatrix{D \ar[r]|{o}^{w} & E}$ and a 2-cell $T(wu)=Tw Tu \Mr{\alpha} Tw Tv=T(wv)$,  that we may take invertible if $u \in \Sigma$. Then, since $T$ is pseudo-fully-faithful, this gives a 2-cell $uw\Mr{} vw$, which is invertible if $u \in \Sigma$. 

\noindent
$\sigma {\bf F2}.$ Given $C \cellpairrdc{u}{\theta}{\eta}{v} D \in \C$, consider $TC \cellpairrdc{Tu}{T\theta}{T\eta}{Tv} TD \in \C'$. Since $T$ is $\sigma$-cofinal, there exists a morphism  $\xymatrix{D \ar[r]|{o}^{w} & E}$ such that $T(w \theta)=Tw T\theta =Tw T\eta =T(w\eta)$. Then, since $T$ is pseudo-fully-faithful, $w \theta= w \eta$. 

\end{proof}

\begin{\prop} \label{prop:Tsubeta}
 Let $P,Q: \A \mr{} \Cat$ be $2$-functors, and $P \Mr{\eta} Q$ a pseudonatural transformation. If $\eta_A$ is full and faithful for each $A \in \A$, then
 the induced $2$-functor \mbox{$\cc{E}l_P \mr{T_{\eta}} \cc{E}l_Q$} (recall \ref{sin:Tsubeta}) is $2$-fully-faithful and the $1$-subcategories given by the cocartesian arrows satisfy $\cart_P = T_\eta^{-1}(\cart_Q)$.
\end{\prop}
\begin{proof}
 Recall the formulas in \ref{sin:Tsubeta}. 
Let $(\eta_A(x),A) \mr{(f,\psi)} (\eta_B(y),B)$, consider then $\eta_B (Pf (x)) \mr{(\eta_f^{-1})_x} Qf \eta_A(x) \mr{\psi} \eta_B(y)$, since
 $\eta_B$ is full and faithful there is a unique $Pf (x) \mr{\varphi} y$ such that $T_{\eta}(f,\varphi) = (f,\psi)$.
 This shows that $T_\eta$ is $2$-fully-faithful (the fact that we have an isomorphism between $2$-cells is trivial). 
 
 To show that $\cart_P = T_\eta^{-1}(\cart_Q)$, note that $\eta_B(\varphi) \circ (\eta_f)_x$ is an isomorphism if and only if $\eta_B(\varphi)$ is so, which since $\eta_B$ is full and faithful happens if and only if $\varphi$ is an isomorphism.
\end{proof}

\begin{\prop} \label{prop:TsubHblah}
 Consider $2$-functors $\A \mr{H} \B \mr{P} \Cat$, and the induced $2$-functor $\cc{E}l_{PH} \mr{T_H} \cc{E}l_P$ as in \ref{sin:TsubH}. 
 If $H$ is $2$-fully-faithful, then so is $T_H$ and the $1$-subcategories given by the cocartesian arrows satisfy $\cart_{PH} = T_H^{-1}(\cart_P)$.
\end{\prop}
\begin{proof}
 It is immediate from the formulas $T_{H}(f,\varphi) = (Hf,\varphi)$, $T_H(\theta)=H\theta$ in \ref{sin:TsubH}.
\end{proof}

\section{Flat pseudofunctors and the main theorems} \label{flat}

\subsection{Flat pseudofunctors} \label{sub:flatpseudo} 

  In this subsection we will consider pseudofunctors between $2$-categories. Though our objective when starting the research that led to this paper was to 
  have results for \mbox{$2$-functors,} it turned out that the correct generality in which to define flat $2$-functors is to consider flat pseudofunctors. 
     For $2$-categories $\A,\B$, we denote by $p\cc{H}om_p(\A,\B)$ the $2$-category of pseudofunctors, pseudonatural transformations and modifications. 
     We refer the reader to the Appendix \ref{appendix} for the complete definitions of these concepts, noting that we will not need the explicit formulas of Definition \ref{def:pseudofunctor} in this section.
 
  Weighted bilimits for pseudofunctors are considered for example in \cite{S2}, \cite{Fiore}, \cite[\S$\,$2]{PKan1}. We recall this notion, with an approach more similar to the one of \S$\,$\ref{sub:epslimits}, and show some basic results that will be needed later. For the sake of simplicity in the exposition, we consider only the case $\sigma=p$, but we note that we could also define $\sigma$-bilimits of pseudofunctors.

  \begin{\de} \label{de:NEparapseudo}
Given pseudofunctors $\A \mr{W} \Cat$, $\A \mr{F} \B$, and $E$ an object of $\B$, we denote $pCones_{}^W(F,E) = p\cc{H}om_{p}(\A,\Cat)(W,\B(E,F-))$. This is the category of $w$-pseudocones (with respect to the weight $W$) for $F$ with vertex $E$. 
 
The bilimit of $F$ weighted by $W$, denoted ${}_{bi} \! \{W,F\}_p$ or more precisely $({}_{bi} \! \{W,F\}_p, \xi)$, is a $w$-pseudocone $\xi$ 
with vertex ${}_{bi} \! \{W,F\}_p$ universal in the sense that 

\begin{equation} \label{eq:ssssparapseudo}
\B(B,{}_{bi} \!  \{W,F\}_p ) \mr{\xi^*} p\cc{H}om_{p}(\A,\Cat)(W,\B(B,F-))
 \end{equation}
 
 $$B \cellrd{f}{\alpha}{g} E \quad \longmapsto \quad W \Mr{\xi} \B(E,F-) \cellrd{f^*}{\alpha^*}{g^*} \B(B,F-) $$
 
\noindent is an equivalence of categories (pseudonatural in the variable $B$).
\end{\de}

Bilimits behave pseudofunctorially respect to pseudonatural transformations: 

\begin{remark} \label{rem:limitsfuntorialparapseudo}
 Let $\A \cellrd{V}{\alpha}{W} \Cat, \quad \A \cellrd{F}{\beta}{G} \B$ be pseudonatural transformations between pseudofunctors. With a similar argument as in Remark \ref{rem:limitsfuntorial} it follows that 
 there are pseudofunctors
 $$(p\cc{H}om_p(\A,\Cat)^{op})_+ \xr{{}_{bi} \! \{-,F\}_p} \B,  
 \hspace{4ex} (p\cc{H}om_p(\A,\B))_+ \xr{{}_{bi} \!\{W,-\}_p} \B,$$
 $$(p\cc{H}om_p(\A,\Cat)^{op} \times p\cc{H}om_p(\A,\B))_+ \xr{{}_{bi} \!\{-,-\}_p} \B.$$
\end{remark}

\noindent where the subscript ``$+$'' indicates the full-subcategories with objects such that the corresponding bilimits exist.

As for $2$-functors, we refer to equivalences in $p\cc{H}om_p(\A,\B)$ as pseudo-equivalences. 
Since pseudofunctors send equivalences to equivalences, we have:

1.\; If $\alpha$ is a pseudo-equivalence, then $\;{}_{bi} \! \{W,F\}_p \xr{{}_{bi} \! \{\alpha,F\}_p}  \;{}_{bi} \! \{V,F\}_p$ is an equivalence.

2.\; If $\beta$ is a pseudo-equivalence, then $\;{}_{bi} \! \{W,F\}_p \xr{{}_{bi} \!\{W,\beta\}_p} \;{}_{bi} \! \{W,G\}_p$ is an equivalence. \qed

\vspace{1ex}

    Note that the definitions of preservation of bilimits (Definition \ref{de:conmutarconlim}), and left exactness (Definition \ref{def:leftexact}) make perfect sense for pseudofunctors.  
From Remark \ref{rem:limitsfuntorialparapseudo}, item 2, it follows:

\begin{corollary} \label{coro:exactoestableporequiv}
  Let $\A \cellrd{F}{\beta}{G} \B$ be a pseudo-equivalence between pseudofunctors. 
  Then any weighted bilimit preserved by $F$ is also preserved by $G$. In particular, $F$ is left exact if and only if $G$ is.
\qed
  \end{corollary}
 
Recall that a $Set$-valued functor is flat when its left Kan extension along the Yoneda embedding is left exact (see for example \cite[\S$\,$VII.5]{MMcL}).
This notion is considered in \cite[\S$\,$6]{K3} for 
$\cc{V}$-enriched categories in general, and in particular for $\cc{V} = \Cat$. However, as it is usually the case (for example with limits), the $\Cat$-enriched version is too strict, and a \emph{relaxed} version is the important notion. 

In Definition \ref{def:2flat} below, we will introduce the notion of \emph{flat} pseudofunctor into $\Cat$. The reader should be aware that if $\A \mr{P} \Cat$ is a $2$-functor 
(as we will consider in \S$\,$\ref{sub:maintheorems}), both Kelly's notion of flatness and ours make sense, but are not at all equivalent. We will always be referring to our notion.

A relaxed notion of Kan extension was already considered in \cite{PKan1}, where it was denoted pseudo Kan extension. We review the main results while, as it is defined by a bicolimit, changing the notation into the one adopted in this paper. We will use (and therefore choose to define) the left bi-Kan extension. 

Let $\C$ be a $2$-category with weighted bicolimits. We will only use the case $\C = \Cat$ in this paper. Given two pseudofunctors 
$\A \mr{P} \C$, $\A \mr{H} \E$, consider the composite \mbox{$\E \mr{h} p\cc{H}om_p(\E^{op},\Cat) \mr{H^*} p\cc{H}om_p(\A^{op},\Cat)$} of the Yoneda embedding \ref{pseudoyoneda}, with the pseudofunctor determined by precomposition with $H$, which we denote $\E(H,-) = H^* \circ h$. We have:

\begin{\de} [{\cite[9.3]{PKan1}}] \label{def:kanforpseudofunctors} 
The left (pointwise) bi-Kan extension of $P$ along $H$ is the pseudofunctor $L = Lan_HP: \E \mr{} \C$ given by the formula $LE = \E(H,E) \;{}_{bi} \! \otimes_p P$ for $E \in \E$, that is $L = (H^* \circ h)  \;{}_{bi} \! \otimes_p P$ (see Remark \ref{rem:limitsfuntorialparapseudo}).
\end{\de}

The pointwise bi-Kan extension has the following important universal property, which can also be considered as the definition of a (not necessarily pointwise) bi-Kan extension:

\begin{\prop} [{\cite[9.6]{PKan1}}] \label{prop:bikanpointwise}
Given pseudofunctors $\A \mr{P} \C$, $\A \mr{H} \E$, for each pseudofunctor $\E \mr{Q} \C$ we have an equivalence 
\begin{equation} \label{eq:kanadjunction}
p\cc{H}om_p(\E,\C)(Lan_HP,Q) \stackrel[\approx]{r}{\longrightarrow} p\cc{H}om_p(\A,\C)(P,QH)
\end{equation}
pseudonatural in $Q$. \qed
\end{\prop}

\begin{remark}\label{rem:unitgeneral} 
 Equation \eqref{eq:kanadjunction} expresses a biadjunction between precomposition with $H$ and $Lan_H$. 
 The unit of this biadjunction consists of a pseudonatural transformation \mbox{$P \Mr{\eta} Lan_HP \circ H$,} which is given by $\eta = r(id_{Lan_HP})$
 in \eqref{eq:kanadjunction}: 
 
 $$\xymatrix{\A \ar[rr]^H \ar[dr]_P && \E \ar[dl]^{Lan_HP }  \\ & \C \ar@{}[u]|{\stackrel{\eta}{\Rightarrow}} } $$
 
 It can be seen (following the $id_{Lan_HP}$ in the chain of equivalences in the proof of \cite[9.6]{PKan1}) that $\eta_A = (\nu_{HA})_A (id_{HA})$, where for each $E \in \E$ we denote by $\nu_E$ the unit of the bicolimit 
 $(Lan_HP)E = \E(H,E) \;{}_{bi} \! \otimes  P$, $\E(H,E) \Mr{\nu_E} \C(P-,(Lan_HP)E)$.
\end{remark}

From Remark \ref{rem:limitsfuntorialparapseudo}, item 2, we have:

\begin{proposition} \label{coro:PyPprima}
 Consider pseudofunctors $\A \xr{P,\,Q} \C$, $\A \mr{H} \E$. If $P$ and $Q$ are pseudo-equivalent, then so are $Lan_H P$  and $Lan_H Q$.
\qed
\end{proposition}

If $H$ is pseudo-fully-faithful (see \S \ref{sub:terminology}, item 7), then the bi-Kan extension is really a (pseudo) extension:

\begin{\prop}[{\cite[9.5]{PKan1}}] \label{prop:410PKan2}
With the notation of Definition \ref{def:kanforpseudofunctors}, if $H$ is pseudo-fully-faithful, then the unit $\eta$ of Remark \ref{rem:unitgeneral} is a pseudo-equivalence (recall that this amounts to each $\eta_A$ being an equivalence of categories). \qed
\end{\prop}

\begin{\de} \label{def:2flat}
Let $\A \mr{P} \Cat$ be a pseudofunctor, consider the Yoneda embedding \ref{pseudoyoneda} \mbox{$\A \mr{h} \cc{H}om_p(\A^{op},\Cat)$,} we denote $P^* = Lan_h P$. 
We say that $P$ is \emph{flat} if $P^*$ is left exact (note that this is well defined by Corollary \ref{coro:exactoestableporequiv}).
Note that, by Proposition \ref{prop:410PKan2}, the following diagram commutes up to pseudo-equivalence:

\begin{equation} \label{eq:PequivPh}
\xymatrix{\A \ar[rr]^<<<<<<<<<<h \ar[dr]_P && \cc{H}om_p(\A^{op},\Cat) \ar[dl]^{P^*}  \\ & \Cat \ar@{}[u]|{\stackrel[\approx]{\eta}{\Rightarrow}} } 
\end{equation}
\end{\de}

\begin{\prop} \label{prop:primermitadkelly62}
 A flat pseudofunctor $\A \mr{P} \Cat$ preserves any finite (weighted) bilimit that exists in $\A$. 
 \end{\prop}

\begin{proof}
It follows immediately from diagram \eqref{eq:PequivPh} and Corollary \ref{coro:exactoestableporequiv} (note that $h$ preserves weighted bilimits by Corollary \ref{comeout}).
\end{proof}

Consider $P,\;h,\;P^*$ as in Definition \ref{def:2flat}. It follows from Remark \ref{rem:limitsfuntorialparapseudo}, item 1, that for a $2$-functor $\A^{op} \mr{F} \Cat$ 
the formula \mbox{$P^*F = \cc{H}om_p(\A^{op},\Cat)(h,F) {}_{bi} \! \otimes_p P$} in Definition \ref{def:kanforpseudofunctors} is pseudo-equivalent by Yoneda to the usual coend formula $P^*F = F  {}_{bi} \! \otimes_p P$ (recall Corollary \ref{coro:coendcomoweighted}).

For a $2$-functor $\A^{op} \mr{F} \Cat$, from the dual case of Proposition \ref{Pcomobicoend} we have $F \approx F \otimes_p h$. Since this pseudo-colimit is computed pointwise by Proposition \ref{prop:claxcolimptoapto}, for $A \in \A$ we have \mbox{$FA \approx F \otimes_p \A(A,-)$.} It follows: 

\begin{\prop} \label{prop:representableareflat}
The bi-Kan extension of a representable $2$-functor $\A(A,-)$ along $h$ can be chosen to be the evaluation $2$-functor $\cc{H}om_p(\A^{op},\Cat) \xr{ev_A} \Cat$. 
Since by Proposition \ref{prop:claxcolimptoapto} the evaluations preserve any weighted pseudolimit, we have in particular that the representable $2$-functors are flat. \qed
\end{\prop}

\subsection{The main theorem} \label{sub:maintheorems}

Let $\C$ be a $2$-category with weighted pseudo-colimits. We will only need the case $\C = \Cat$ in this paper.

\begin{remark} \label{sin:kanpara2functors}
Consider a $2$-functor $\A \mr{P} \C$, and a $2$-functor $\A \mr{H} \E$. Note that we can compute the bi-Kan extension $L = Lan_HP$ of Definition \ref{def:kanforpseudofunctors} as a $2$-functor $\E \mr{L} \C$. The definition of $L$ is given by the formula $LE = \E(H,E) \;{}_{bi} \! \otimes_p P$, but we can compute it by the equivalent pseudo-colimit $LE = \E(H,E) \otimes_p P$. \cqd
\end{remark}
 
\begin{remark} \label{rem:kanconmutaconcolim}
 Consider $2$-functors $\cc{I} \mr{F} \cc{H}om_p(\A,\C)$,  and $\A \mr{H} \E$. From the dual of Proposition \ref{sin:interchange} we have the equation
$\E(H,E) \otimes_p \coLim{i \in \cc{I}}{F_i} = \coLim {i \in \cc{I}}{\E(H,E) \otimes_p F_i}
$, which together with the fact that $\sigma$-colimits are computed pointwise, implies immediately  
the equation $Lan_H(\coLim{i \in \cc{I}}{F_i})E = (\coLim{i \in \cc{I}}{Lan_H F_i})E$. That is, the left bi-Kan extension commutes with $\sigma$-colimits. \cqd
\end{remark}

With this, and using again that  $\sigma$-colimits in $\cc{H}om_p(\A,\Cat)$ are computed pointwise,
we have the following immediate corollary of Theorem \ref{teo:conmutan}.

\begin{proposition} \label{coro:colimitdeexactosesexacto}
 A $\sigma$-filtered $\sigma$-colimit in $\cc{H}om_p(\A,\Cat)$ of left exact $2$-functors is left exact. \qed  
\end{proposition}

From Remark \ref{rem:kanconmutaconcolim} and Proposition~\ref{coro:colimitdeexactosesexacto} it follows (all $\sigma$-colimits below are considered in $\cc{H}om_p(\A,\Cat)$): 

\begin{corollary} \label{coro:colimitdeplayosesplayo}
 A $\sigma$-filtered $\sigma$-colimit of flat $2$-functors is flat. In particular, by Proposition \ref{prop:representableareflat}, a $\sigma$-filtered $\sigma$-colimit of representable $2$-functors is flat.
 \qed
\end{corollary}

\begin{lemma} \label{lema:gammapengammapestrella}
 Let $\A \mr{P} \Cat$, $\A \mr{H} \E$, $\E \mr{L} \Cat$ as in \ref{sin:kanpara2functors}. 
 Consider the $1$-subcategories $\cart_P$ of $\cc{E}l_P$, and $\cart_{L}$ of $\cc{E}l_{L}$ as in Definition \ref{def:gammaP}. Then there exists a canonical $2$-functor 
 $$T: \cc{E}l_P \mr{} \cc{E}l_{L}$$ 
  satisfying (the dual of) axiom $\sigma {\bf C0}$ in Definition \ref{def:initial}. If $H$ is $2$-fully-faithful, then so is $T$ and $\cart_P = T^{-1} (\cart_L)$.
\end{lemma}
\begin{proof}
 $T$ is defined as the composition of the $2$-functors $\cc{E}l_P \mr{T_\eta} \cc{E}l_{LH} \mr{T_H} \cc{E}l_{L}$ considered in propositions \ref{prop:Tsubeta} and \ref{prop:TsubHblah}, where $\eta$ is the pseudonatural transformation of Remark \ref{rem:unitgeneral}. Then we have the formula $T(x,A) = (\eta_A(x),HA)$. Let $(c,E) \in \cc{E}l_{L}$, we will show that there is an arrow in $\cart_L$ of the form $(\eta_A(x),HA) \xr{(\theta,id)} (c,E)$.
 
 We have $c \in LE = \E(H-,E) \otimes_p P$, then by Lemma \ref{lema:elemenplimconpeso} there exist $A \in \A$, $HA \mr{\theta} E$ and $x \in PA$ such that $(\nu_E)_A(\theta)(x) = c$.
 
 We consider the following diagram, which commutes by definition of $L$ on the arrow $\theta$ (see Remark  \ref{rem:limitsfuntorial})
 $$
 \xymatrix{ \E(HA,E) \ar[rr]^{(\nu_E)_A} && \Cat(PA,LE) \\
 \E(HA,HA) \ar[u]^{\theta_*} \ar[rr]^{(\nu_{HA})_A} && \Cat(PA,LHA) \ar[u]_{(L\theta)_*}  }
 $$
 Element chasing $id_{HA}$ and then evaluating at $x$, we have the equality
$$
c = (\nu_E)_A(\theta)(x) = L\theta \circ (\nu_{HA})_A (id_{HA}) (x) \stackrel{\ref{rem:unitgeneral}}{=} L\theta (\eta_A(x)),
$$
\noindent which expresses the fact that $(\theta,id)$ is an arrow of $\cc{E}l_{L}$ as desired.

If $H$ is $2$-fully-faithful, by Proposition \ref{prop:410PKan2} each $\eta_A$ is full and faithful and then by Propositions \ref{prop:Tsubeta} and \ref{prop:TsubHblah} both $T_\eta$ and $T_H$ are $2$-fully-faithful and \mbox{$\cart_P = T_{\eta}^{-1} (\cart_{LH}) =  T^{-1} (\cart_L)$}.
\end{proof}

Using Proposition \ref{prop:initialencofiltered} for the $2$-functor $T^{op}: \cc{E}l_P^{op} \mr{} \cc{E}l_{L}^{op}$ it follows

\begin{corollary} \label{coro:gammapengammapestrella}
 Under all the hypothesis of Lemma \ref{lema:gammapengammapestrella} (including $H$ 2-fully-faithful), if $\cc{E}l_{L}$ is $\sigma$-cofiltered (with respect to $\cart_L$), then $\cc{E}l_P$ is $\sigma$-cofiltered (with respect to $\cart_P$). 
 \qed
\end{corollary}

It is a classical result (see for example \cite[\S$\,$VII.6]{MMcL}) that every flat $\cc{S}et$-valued functor is a filtered colimit of representable functors, that, as far as we know (see \cite[(6.4)]{K3}) has no known generalization to other base categories. Here we extend this result to $2$-dimensional category theory. Note that from Theorem \ref{th:main} it follows that if a $2$-functor $\A \mr{P} \Cat$ is pseudo-equivalent to any $\sigma$-filtered $\sigma$-colimit of representable $2$-functors, then the $\sigma$-colimit in its canonical expression \ref{Pcomosigmabicolimit} (whose diagram is actually in $Hom_s(\A,\Cat)$, see remark \ref{rem:diagenHoms}) is also $\sigma$-filtered.

\begin{theorem}\label{th:main}
 Let $\A \mr{P} \Cat$ be a $2$-functor. Then the following are equivalent.
 
 \begin{itemize}
  \item[(i)] $\cc{E}l_P$ is $\sigma$-cofiltered with respect to the family $\cart_P$ of cocartesian arrows.
  \item[(ii)] $P$ is equivalent to a $\sigma$-filtered $\sigma$-colimit of representable \mbox{$2$-functors in $\cc{H}om_p(\A,\Cat)$.}
  \item[(iii)] $P$ is flat.
 \end{itemize}
\end{theorem}

\begin{proof}
 $(i) \Rightarrow (ii)$ follows immediately by the canonical expression \ref{Pcomosigmabicolimit}. $(ii) \Rightarrow (iii)$ holds by Corollary \ref{coro:colimitdeplayosesplayo} (note that flatness is preserved by pseudo-equivalence by Corollary \ref{coro:PyPprima}). 
 $(iii) \Rightarrow (i)$: If $P^*$ is left exact, by Propositions \ref{prop:claxcolimptoapto} and \ref{prop:exactothenfiltering}, $\cc{E}l_{P^*}$ is $\sigma$-cofiltered with respect to $\cart_{P^*}$. Then
 $(i)$ follows by Corollary \ref{coro:gammapengammapestrella}.
\end{proof}

\begin{remark} \label{rem:flatiffbicolimit}
Note that, since $\sigma$-bicolimits are defined up to equivalence, we can say that the flat \mbox{$2$-functors} are exactly the 
$\sigma$-filtered $\sigma$-bicolimits of representable $2$-functors.
\end{remark}

Combining Proposition \ref{prop:primermitadkelly62}, Proposition \ref{prop:exactothenfiltering} and the implication $(i) \Rightarrow (iii)$ in the theorem above, it follows:

\begin{\prop} \label{prop:flatsiiexact} 
If $\A$ is finitely complete, then a $2$-functor $\A \mr{P} \Cat$ is flat if and only if it is left exact. \qed
\end{\prop}

\bibliographystyle{unsrt}

\begin{appendices}
\numberwithin{equation}{section}

\section{The main theorem for pseudofunctors} \label{appendix}

We will now prove a generalization of Theorem \ref{th:main} and Proposition \ref{prop:flatsiiexact} to $\Cat$-valued pseudofunctors $\A \mr{P} \Cat$ (with $\A$ still a $2$-category). We will prove it by applying those results to the $2$-functor $\widetilde{P}$ associated to the pseudofunctor $P$. 
We note that, while it is tempting to try and develop this generalization for $\A$ a bicategory, the computations with the bicategory $\cc{E}l_P$ are more complicated than in the 2-category case and, more fundamentally, as far as we know the fact that $\widetilde{P}$ is pseudo-equivalent to $P$ has not been shown in this case. The interested reader can check this possibility, as we may in the future.

We begin by giving the explicit definition of the $2$-category $p\cc{H}om_p(\A,\B)$ considered in Section \ref{flat}. We will use the explicit formulas defining pseudofunctors and pseudonatural transformations.  
  We refer the reader to \cite[\S$\,$2]{PKan1}, \cite[\S3]{Fiore}, \cite[\S1]{tesisEmi} among other choices for a more expanded description of the equations 
  below.
  
  \begin{\de} \label{def:pseudofunctor}
   Let $\A, \B$ be $2$-categories. A \emph{lax functor} $\A \mr{F} \B$ is given by the following data:
   
   \begin{itemize}
    \item[-] For each object $A \in \A$, an object $FA \in \B$.
    \item[-] For each hom-category $\A(A,B)$, a functor $\A(A,B) \xr{F_{A,B}} \B(FA,FB)$. Whenever possible we will abuse the notation $F_{A,B}$ by $F$.
    \item[-] For each object $A \in \A$, an invertible 2-cell $\alpha_A^F : id_{FA} \Mr{} F(id_A)$.
    \item[-] For each triplet of objects $A,B,C \in \A$, a natural transformation
    
    $$\xymatrix{ \A(B,C) \times \A(A,B) \ar@{}[dr]|{\alpha^F \Downarrow} \ar[d]^{\circ} \ar[r]^>>>>{F \times F} & \B(FB,FC) \times \B(FA,FB) \ar[d]^{\circ} \\
		    \A(A,C) \ar[r]^{F} & \B(FA,FC)}.$$
    
    This natural transformation is given, for each configuration $A \mr{f} B \mr{g} C$ by $2$-cells of $\B$, $FgFf \Mr{\alpha^F_{f,g}} F(gf)$. These data are subject to the axioms
   \end{itemize}

   \begin{tabular}{rll}
    {\bf LF0}. & For each $A \mr{f} B$, & $\alpha_{f,id_B}^F \circ (\alpha_B^F Ff) = Ff = \alpha_{id_A,f}^F \circ (Ff \alpha_A^F)$. \\
    {\bf LF1}. & For each $A \mr{f} B \mr{g} C \mr{h} D$, & $\alpha_{gf,h}^F \circ (Fh \alpha_{f,g}^F) = \alpha_{f,hg}^F \circ (\alpha_{g,h}^F Ff)$.
   \end{tabular}
   
   A \emph{lax natural transformation} $\theta$ between lax functors $\A \mrpair{F}{G} \B$ is 
   given by families  $\{FA \mr{\theta_A} GA\}_{A \in \A}$,  $\{Gf \theta_A \Mr{\theta_f} \theta_B Ff\}_{A\mr{f}B \in \A}$ satisfying the equations (cf. \S$\,$\ref{sub:terminology}, item \ref{laxoplaxnatural}):
   
   \bigskip

   \begin{tabular}{rll}
{\bf LN0}. & For all $A\in \A$, & ${\theta}_{id_A} \circ \alpha_A^G\theta_A = \theta_A \alpha_A^F$. \\
{\bf LN1}. & For all $A \mr{f} B \mr{g} C \in \A$, & $\theta_{gf} \circ \alpha_{f,g}^G\theta_C = \theta_C\alpha_{f,g}^F \circ \theta_g Ff \circ Gg \theta_f$. \\
{\bf LN2}. & For all $A \cellrd{f}{\gamma}{g} B \in \A$, & $\theta_B F\gamma \circ \theta_f = \theta_g \circ G\gamma \theta_A$. 
   \end{tabular}


   A \emph{pseudofunctor} is a lax functor $F$ such that the structural $2$-cells $\alpha_A^F$, $\alpha_{f,g}^F$ are all invertible. A \emph{pseudonatural transformation} between lax functors is a lax natural transformation such that the structural $2$-cells $\theta_f$ are all invertible. Modifications are defined as for $2$-functors.
  \end{\de}

Now we extend the Definition \ref{def:gammaP} of the $2$-category of elements to the case of pseudofunctors. We note that this construction is considered in \cite{Data}, and with greater generality in \cite[3.3.3]{Buckley}, where a theory of fibred $2$-categories is developed, corresponding to pseudofunctors with values in $2$-categories. An idea that is useful to have in mind is that $\Cat$-valued pseudofunctors are the $2$-dimensional analogous to the ``discrete'' $1$-dimensional fibrations.

\begin{\de}
 Let $\A \mr{P} \Cat$ be a pseudofunctor. $\cc{E}l_P$ is the $2$-category with objects, morphisms and $2$-cells described exactly as in Definition \ref{def:gammaP}, but now the structural $2$-cells of the pseudofunctor appear in the formulas for composition and identities:
 
 For $(x,A) \xr{(f,\varphi)} (y,B) \xr{(g,\psi)} (z,C)$, the composition is given by the formula $$(g,\psi)(f,\varphi) = (gf,\psi Pg(\varphi) (\alpha_{f,g}^P)^{-1}_x).$$ Identity morphisms are given by \mbox{$(x,A) \xr{(id_A,(\alpha_A^P)^{-1}_{x})} (x,A)$.} $2$-cells are composed as in $\A$.

 As for $2$-functors, we consider the $1$-subcategory $\cart_P$ of $\cc{E}l_P$ whose arrows are $(f,\varphi)$ with $\varphi$ an isomorphism. 
 \end{\de} 

\begin{remark}
 The fact that $(x,A) \xr{(id_A,(\alpha_A^P)^{-1}_x)} (x,A)$ are identities follows from axiom {\bf LF0} in Definition \ref{def:pseudofunctor}. The fact that the composition of morphisms is associative follows from axiom {\bf LF1}. In both cases the naturality of the structural $2$-cells $\alpha_A^P$, $\alpha_{f,g}^P$ respectively is used. The computations are somewhat lengthy but straightforward so we omit them.
\end{remark}

We also extend the results of \ref{sin:Tsubeta} and Proposition \ref{prop:Tsubeta} to pseudofunctors.

\begin{\prop} \label{prop:Tsubetaparapseudo}
 For each lax natural transformation $P \Mr{\eta} Q$ between $\Cat$-valued \mbox{pseudofunctors,} there is an induced $2$-functor $\cc{E}l_P \mr{T_{\eta}} \cc{E}l_Q$ given by the same formulas in \ref{sin:Tsubeta}.

\end{\prop}

\begin{proof}
 To show that $T_{\eta}$ preserves composition of morphisms strictly, consider \mbox{$(x,A) \mr{(f,\varphi)} (y,B) \mr{(g,\psi)} (z,C)$} in $\cc{E}l_P$, then the equation we have to show is
 $$(gf, \eta_C(\psi Pg(\varphi) (\alpha_{f,g}^P)^{-1}_x) \ (\eta_{gf})_x) = (gf, \eta_C(\psi) \ (\eta_g)_y \ Q(g)(\eta_B(\varphi)(\eta_f)_x) \ (\alpha_{f,g}^Q)^{-1}_{\eta_A(x)} )$$
 This equation follows at once from axiom {\bf LN1} in Definition \ref{def:pseudofunctor} using the naturality of $\eta_g$ with respect to the arrow $\varphi$. The fact that $T_{\eta}$ preserves identities follows immediately from axiom {\bf LN0}. The rest of the verifications of the $2$-functoriality are straightforward and identical to the case of $2$-functors so we omit them.
\end{proof}

 We note that the formulas in \ref{sin:Tsubeta} are the same formulas 
 of \cite[3.3.12]{Buckley}, where it is stated (for a pseudonatural transformation $\eta$, though the same proof would work for lax natural instead) that $T_{\eta}$ is a morphism of bicategories. In our case computations are simpler, and we have $2$-functoriality instead.
Proposition \ref{prop:Tsubeta} holds for pseudofunctors with exactly the same proof:

\begin{\prop} \label{prop:Tsubetaparapseudo2}
  Let $P,Q: \A \mr{} \Cat$ be pseudofunctors, and $P \Mr{\eta} Q$ a pseudonatural transformation. If $\eta_A$ is full and faithful for each $A \in \A$, then the
 $2$-functor $\cc{E}l_P \mr{T_{\eta}} \cc{E}l_Q$ of Proposition \ref{prop:Tsubetaparapseudo}
 is $2$-fully-faithful and the $1$-subcategories given by the cocartesian arrows satisfy $\cart_P = T_\eta^{-1}(\cart_Q)$. \qed
\end{\prop}

We now recall (see \cite[4.2]{Power}, or the nLab website on pseudofunctors) the construction of the  $2$-functor $\A \mr{\widetilde{P}} \Cat$ associated to a pseudofunctor  $\A \mr{P} \Cat$. We state only the facts that we will need.

Given a pseudofunctor $\A \mr{P} \Cat$, there is a $2$-functor $\A \mr{\widetilde{P}} \Cat$ and an equivalence in $p\cc{H}om_p(\A,\Cat)$ between $P$ and $\widetilde{P}$, i.e. a pseudo-equivalence $P \Mr{\eta} \widetilde{P}$. The description of $\widetilde{P}$ on objects is as follows, $\widetilde{P}B$ is the category with pairs $(f,x)$ as objects, where $A \mr{f} B \in \A$ and $x \in PA$, and arrows $(f,x) \mr{\varphi} (f',x')$ given by an arrow $Pf(x) \mr{\varphi} Pf'(x')$ in $PB$. For $B \mr{g} B'\in \A$, we have $\widetilde{P}g(f,x) = (gf,x)$. The definition of $PA \mr{\eta_A} \widetilde{P}A$ on objects is $\eta_A(x) = (id_A,x)$, and for the pseudo-inverse $\widetilde{P} \Mr{\eps} P$ we have $\eps_B(f,x) = Pf(x)$. 

When applied to $\widetilde{P}$, Theorem \ref{th:main} yields:

\begin{theorem}\label{th:parapseudo}
 Let $\A \mr{P} \Cat$ be a pseudofunctor.  Then the following are equivalent.
 
 \begin{itemize}
  \item[(i)] $\cc{E}l_P$ is $\sigma$-cofiltered with respect to the family $\cart_P$ of cocartesian arrows.
  
  \vspace{-.2cm}
  
  \item[(ii)] $P$ is a $\sigma$-filtered $\sigma$-bicolimit of representable \mbox{$2$-functors} in $p\cc{H}om_p(\A,\Cat)$.
  
  \vspace{-.2cm}
  
  \item[(iii)] $P$ is flat. 
 \end{itemize}
\end{theorem}
\begin{proof}
We will show the equivalence of each of the items 
above with the corresponding statement
of Theorem \ref{th:main} for the $2$-functor $\widetilde{P}$:

\vspace{1ex}

\noindent
\emph{(i)}  By Proposition \ref{prop:Tsubetaparapseudo2} we have induced $2$-functors $\cc{E}l_P \mr{T_{\eta}} \cc{E}l_{\widetilde{P}}$, $\cc{E}l_{\widetilde{P}} \mr{T_{\eps}} \cc{E}l_P$, both \mbox{$2$-fully-faithful} and satisfying $\cart_P = T_\eta^{-1}(\cart_{\widetilde{P}})$, $\cart_{\widetilde{P}} = T_\eps^{-1}(\cart_P)$. In order to show that $\cc{E}l_P$ is $\sigma$-cofiltered if and only if $\cc{E}l_{\widetilde{P}}$ is so, by Proposition \ref{prop:initialencofiltered} it suffices to show axiom $\sigma {\bf C0}$ for these $2$-functors.
  
\noindent  
\begin{tabular}{lp{13.5cm}}
 For $T_{\eta}$: & given $((f,x),B)$ in $\cc{E}l_{\widetilde{P}}$, where $A \mr{f} B \in \A$ and $x \in PA$, consider \mbox{$T_{\eta}(x,A) = ((id_A,x),A) \xr{(f,id_{Pf(x)})} ((f,x),B)$.} \\
For $T_{\eps}$: & given $(x,A)$ in $\cc{E}l_P$, consider $T_{\eps}((id_A,x),A) = (P(id_A)(x),A) \mr{(id_A,\varphi)} (x,A)$, where $\varphi$ is the isomorphism $\varphi: P(id_A)P(id_A)(x) \xr{P(id_A)(\alpha_A^P)_x^{-1}} P(id_A)(x) \xr{(\alpha_A^P)_x^{-1}} x$.
 \end{tabular}
  
\vspace{1ex}

\noindent
\emph{(ii)} Immediate from the pseudo-equivalence $P \Mr{\eta} \widetilde{P}$ (recall Remark \ref{rem:flatiffbicolimit}).

\vspace{1ex}

\noindent
\emph{(iii)} Immediate from Corollary \ref{coro:PyPprima}.
\end{proof}

We end the paper showing that, in the presence of finite bilimits, flat pseudofunctors coincide with left exact ones.

\begin{\prop}
If $\A$ is finitely complete, then a pseudofunctor $\A \mr{P} \Cat$ is flat if and only if it is left exact. 
\end{\prop}
\begin{proof}
By Corollary \ref{coro:exactoestableporequiv}, $P$ is left exact if and only if $\widetilde{P}$ is so. By Corollary \ref{coro:PyPprima}, $P$ is flat if and only if $\widetilde{P}$ is so. Then the proposition follows from Proposition \ref{prop:flatsiiexact} applied to $\widetilde{P}$. 
\end{proof}
\end{appendices}
\end{document}